\documentclass[a4paper, 13pt]{article}

\usepackage[T1]{fontenc}
\usepackage{lmodern}   
\usepackage{xcolor}

\usepackage{csquotes}

\usepackage{ifthen}
\usepackage{amsmath,amsthm,amssymb,amsfonts}
\usepackage{mathtools}
\usepackage{tensor}
\usepackage{tikz-cd}
\usepackage{esint}
\usepackage{mathrsfs}
\usepackage{thmtools}
\usepackage{slashed}
\usepackage{mathabx}

\usepackage{bbm}
\usepackage{microtype}
\usepackage{accents}

\usepackage{enumitem}
\usepackage[toc,page]{appendix}
\usepackage{imakeidx}
\usepackage{derivative}

\usepackage{xstring}

\usepackage{emptypage}
\usepackage[a4paper,	bindingoffset=0in,		left=1.2in,		right=1.2in,	top=1in,	bottom=1in,	footskip=.25in]{geometry}
\usepackage{setspace}
\usepackage[backend=biber,style=alphabetic,sorting=ynt,sortcites]{biblatex}
            \usepackage{crossreftools}
\usepackage{hyperref}
\usepackage{cleveref}

\allowdisplaybreaks[1]

\theoremstyle{plain}

\usepackage[english]{babel}

\newtheorem{theorem}{Theorem}[section]
\newtheorem*{theorem*}{Theorem}

\newtheorem{cor}[theorem]{Corollary}
\newtheorem{lem}[theorem]{Lemma}
\newtheorem{prop}[theorem]{Proposition}
\theoremstyle{definition}

\newtheorem{ex}[theorem]{Example}
\newtheorem{exs}[theorem]{Examples}
\newtheorem{dfn}[theorem]{Definition}

\newtheorem{rem}[theorem]{Remark}
\newtheorem{rems}[theorem]{Remarks}
\newtheorem{ques}[theorem]{Question}

\theoremstyle{remark}

\newtheoremstyle{linked}
  {}
  {}
  {\itshape}
  {}
  {\bfseries}
  {.}
  { }
{\hyperlink{InternalLink:\thislink}{\thmname{#1} \thmnumber{#2}}\thmnote{ (#3)}}
\theoremstyle{linked}

\newtheorem{innlinkthm}[theorem]{Theorem} 
 \NewDocumentEnvironment{linkthm}{m o}
   { 
   \IfNoValueTF{#2}{\def\thislink{#1}\begin{innlinkthm}\label{thm:#1}}{\def\thislink{#1}\begin{innlinkthm}[#2]\label{thm:#1}}
   }
   {\end{innlinkthm}}

  \newtheorem{innlinkprop}[theorem]{Proposition} 
  \NewDocumentEnvironment{linkprop}{m o}
   { 
   \IfNoValueTF{#2}{\def\thislink{#1}\begin{innlinkprop}\label{thm:#1}}{\def\thislink{#1}\begin{innlinkprop}[#2]\label{thm:#1}}
   }
   {\end{innlinkprop}}

\NewDocumentEnvironment{linkproof}{m o}
  {
    \IfNoValueTF{#2}
      {
        \hypertarget{InternalLink:#1}{
          \paragraph{Proof of \Cref{thm:#1}.}
        }
      }
      {
        \hypertarget{InternalLink:#1}{}
        \hypertarget{InternalLink:#2}{
          \paragraph{Proof of \Cref{thm:#1} and \Cref{thm:#2}.}
        }
      }
  }
  {
    \hfill$\qed${\parfillskip=0pt\par}
  }

\newcommand{\RNum}[1]{\uppercase\expandafter{\romannumeral #1\relax}}

\makeatletter
\providecommand*{\twoheadrightarrowfill@}{%
  \arrowfill@\relbar\relbar\twoheadrightarrow}
\providecommand*{\twoheadleftarrowfill@}{%
  \arrowfill@\twoheadleftarrow\relbar\relbar}
\providecommand*{\xtwoheadrightarrow}[2][]{%
  \ext@arrow 0579\twoheadrightarrowfill@{#1}{#2}}
\providecommand*{\xtwoheadleftarrow}[2][]{%
  \ext@arrow 5097\twoheadleftarrowfill@{#1}{#2}}
\newcommand\setItemnumber[1]{\setcounter{enum\romannumeral\@enumdepth}{\numexpr#1-1\relax}}
\makeatother

\newcommand\norm[1]{\left\lVert#1\right\rVert}
\newcommand\Crefitem[2]{\nameCref{#1} \hyperref[#1:#2]{\ref*{#1}.\ref*{#1:#2}}}
\newcommand\refitem[2]{\hyperref[#1:#2]{\ref*{#1}.\ref*{#1:#2}}}

\NewDocumentCommand{\MyCref}{m o}{%
  \IfNoValueTF{#2}
    {\Cref{#1}} 
    {\nameCref{#1} \hyperref[#1:#2]{\ref*{#1}.\ref*{#1:#2}}}
}

\DeclareMathOperator\supp{supp}

\DeclareMathOperator\ad{ad}

\DeclareMathOperator\im{Im}

\newcommand{\R}{\mathbb{R}}
\newcommand{\Q}{\mathbb{Q}}
\newcommand{\N}{\mathbb{N}}
\newcommand{\C}{\mathbb{C}}
\newcommand{\Z}{\mathbb{Z}}

\newcommand{\cA}{\mathcal{A}}
\newcommand{\cB}{\mathcal{B}}
\newcommand{\cC}{\mathcal{C}}

\newcommand{\cK}{\mathcal{K}}
\newcommand{\cL}{\mathcal{L}}

\newcommand{\Rpt}{\mathbb{R}_+^\times}

\newcommand\bb[1]{[\![ #1]\!]}

\AtEveryBibitem{\clearfield{url}}

\usepackage{prerex}
\usetikzlibrary{fit}

\newcommand{\Id}{\mathrm{Id}}
\newcommand{\DO}{\mathrm{DO}}

\newcommand\Rd[1]{\cA_{#1}}

\newcommand{\singsupp}{\mathrm{singsupp}}

\newcommand{\omegahalf}{\Omega^\frac{1}{2}}
\renewcommand{\Re}{\mathrm{Re}}

\usepackage{derivative}

\begin{document}
\title{Abstract maximal hypoellipticity and applications}
\author{Omar Mohsen\thanks{{
Université Paris Cité, Sorbonne Université, CNRS, IMJ-PRG, F-75013 Paris, France, \href{mailto:omar.mohsen@imj-prg.fr}{\texttt{omar.mohsen@imj-prg.fr}}}}}
\date{}
\maketitle
\begin{abstract}
We prove an abstract theorem of maximal hypoellipticy showing that in an abstract calculus under some natural assumptions, an operator is maximally hypoelliptic if and only if its principal symbol is left invertible.
We then show that our theorem implies various known results in the literature like regularity theorem for elliptic operators, Helffer and Nourrigat's resolution of the Rockland conjecture, Rodino's theorem on regularity of operators on products of manifolds, and our resolution of the Helffer-Nourrigat conjecture.
Other examples like our resolution of the microlocal Helffer-Nourrigat conjecture will be given in a sequel to this paper.

Our arguments are based on the theory of $C^*$-algebras of Type I.

\end{abstract}
\setcounter{tocdepth}{2} 
\tableofcontents

\section*{Introduction}
Let $D:C^\infty(M)\to C^\infty(M)$ be a linear differential operator on a smooth manifold $M$. We say that $D$ is hypoelliptic if for any distribution $u$ on $M$, $\singsupp(Du)=\singsupp(u)$.
A well-known easy result is that if $D$ is a constant coefficient homogeneous differential operator on $\R^n$, then $D$ is hypoelliptic if and only if it is elliptic.
Replacing $\R^n$ by a simply connected graded nilpotent Lie group $G$, Rockland \cite{RocklandConj} conjectured a necessary and sufficient condition for hypoellipticity of homogeneous left-invariant differential operators on $G$.
This conjecture was proved by Helffer and Nourrigat (sufficiency part) and by Beals (necessity part). 
More precisely, they prove the following
 \begin{theorem}[\cite{BealsRocklandConjNecessary,HelfferRockland}]\label{intro_thm_main}
 Let $G$ be a simply connected graded nilpotent Lie group, $D$ a left-invariant differential operator on $G$ which is homogeneous of degree $k\in \N$ with respect to the grading on $G$. The following are equivalent \begin{enumerate}
 \item The differential operator $D$ is hypoelliptic 
  \item For any irreducibly non-trivial unitary representation of $G$, $\pi(D)$  is injective on smooth vectors.
  \item For any differential operator $D'$ (not necessarily left-invariant) homogeneous of degree $\leq k$, and $K\subseteq G$ compact subset, there exists a constant $C>0$ such that 
      \begin{equation}\label{eqn:maxi_hypo_intro}\begin{aligned}
        \norm{D'(f)}_{L^2(G)}\leq C(\norm{f}_{L^2(G)}+\norm{D(f)}_{L^2(G)}),\quad \forall f\in C^\infty_c(G)\text{ such that } \supp(f)\subseteq K.       
      \end{aligned}\end{equation}
 \end{enumerate}
 \end{theorem}
Inequality \eqref{eqn:maxi_hypo_intro} is called the maximal hypoellipticity inequality.
The proof of $1\implies 2$ by Beals \cite{BealsRocklandConjNecessary} is short and simple.
On the other hand, the proof of the other implications by Helffer and Nourrigat is much more involved and requires a detailed study of $\pi(D)$ for all unitary representations of $G$ based on Kirillov's orbit method. 
Since Kirillov's orbit method shows that all irreducible unitary representations are induced from one dimensional representations of subgroups of $G$, 
Helffer and Nourrigat's proof proceeds by carefully studying the behavior of $\pi(D)$ and its norm acting on Sobolev spaces as one induces representations from subgroups. 

In this article, we present a $C^*$-algebra proof of Helffer and Nourrigat's result which avoids any detailed study of the structure of the unitary representations of $G$.
Instead, it relies on the fact that the group $G\rtimes \Rpt$ is of Type \RNum{1} where $\Rpt$  acts on $G$ by graded dilations. 
There are many examples of Lie groups of Type I. All connected nilpotent lie groups and real semisimple Lie groups are of Type I, see \cite{Dixmier}.
Solvable Lie groups are not always of Type I. Dixmier \cite{DixmierSolv} showed that a large class of solvable groups are of Type I which includes $G\rtimes \R_+^\times$. 
More generally, a simple sufficient and necessary condition for simply connected solvable Lie groups to be of Type \RNum{1} was given by Auslander and Kostant \cite{AuslanderKostant}.

Our proof is quite general and applies to a broader range of situations than those considered by Helffer and Nourrigat.
We establish an abstract theorem showing that, given a abstract pseudo-differential calculus equipped with a principal symbol and satisfying the following three main assumptions:
\begin{enumerate}
\item Operators of order $\leq 0$ act as bounded operators on suitable Hilbert spaces.
\item The $C^*$-algebra generated by operators of order $0$ is of Type~\RNum{1}.
\item The calculus contains a family of operators $(v_k)_{k\in \R}$ such that $v_k$ is of order $k$, the family $v_k$ varies continuously in $k$, and $v_0$ is invertible.
\end{enumerate}
Then maximal hypoellipticity in this calculus is equivalent to the injectivity of the principal symbol.
In the case of Helffer and Nourrigat's theorem, the first assumption follows from a classical argument using Cotlar and Stein lemma. The second is that $G\rtimes \Rpt$ is of Type \RNum{1}. The third is straightforward.

Another application of our theorem is a resolution of the microlocal Helffer-Nourrigat conjecture which will be presented in a separate article. Our previous work on the Helffer-Nourrigat conjecture \cite{MohsenMaxHypo} is also a special case of our theorem.

Abstract pseudo-differential calculi have been considered by many people, see for example the work of Connes-Moscovici \cite{ConnesMoscoviciLocalIndexFormula}, Higson \cite{HigsonConnesMoscovici} and  Guillemin \cite{GuilleminNewProofWeylFormula}.
The main novelty of our work lies in the use of $C^*$-algebras of Type~I. 
To motivate this approach, we briefly explain why $C^*$-algebras naturally arise in the proof of \MyCref{intro_thm_main}.

Historically, $C^*$-algebras were introduced by Gelfand in order to give a more conceptual proof of Wiener’s theorem~\cite{WienerFourierTransformCircle}. 
Gelfand’s argument relies on the following classical result in the theory of $C^*$-algebras.

\begin{theorem}[Gelfand]\label{thm:}
Let $A$ be a unital $C^*$-algebra. 
An element $a \in A$ is invertible if and only if for every irreducible representation $\pi$ of $A$ acting on a Hilbert space $L^2(\pi)$, the operator
$
\pi(a) : L^2(\pi) \to L^2(\pi)
$
is invertible.
\end{theorem}

Let us now return to the proof of \MyCref{intro_thm_main}. 
The strategy is to embed the operator $D$ into a pseudo-differential calculus and to consider the operator $v_{-k} D$, where $v_{-k}$ is a suitably chosen operator of order $-k$. 
The operator $v_{-k} D$ then has order~$0$.

The goal is to show that $v_{-k} D$ is invertible on $L^2(G)$. 
Once this is established, maximal hypoellipticity estimates follow. 
Let $A$ denote the $C^*$-algebra obtained as the closure of operators of order~$0$. 
Gelfand’s theorem allows us to reduce the invertibility of $v_{-k} D$ to the study of its images under irreducible representations of~$A$. 
In our setting, these representations arise from the non-trivial unitary representations of the group~$G$.

At this point, one may ask where the Type~I assumption enters the argument. 
In Gelfand’s theorem, one requires that $\pi(a)$ be invertible as an operator on the entire Hilbert space $L^2(\pi)$. 
However, in \MyCref{intro_thm_main}, the assumption is weaker: one only requires that $\pi(D)$ be injective on the space of smooth vectors.

The key observation of this article is that, for $C^*$-algebras of Type~I, it is possible to bridge this gap. 
More precisely, the Type~I condition allows one to reduce injectivity on the whole Hilbert space to injectivity on smooth vectors. 
This reduction relies on a counting argument involving the Fredholm index and on the following theorem, proved in \MyCref{thm:type_I_inv}.

\begin{theorem}
Let $A$ be a unital $C^*$-algebra of Type~I. 
An element $a \in A$ is invertible if and only if, for every irreducible representation $\pi$ of~$A$, the implication
\[
\pi(a) \text{ is Fredholm } \;\Longrightarrow\; \pi(a) \text{ is invertible}
\]
holds.
\end{theorem}

We hope that this work will draw the attention of researchers in linear PDEs to the importance of $C^*$-algebras of Type~I and to their usefulness in analytic problems of this kind.


We end by remarking that we use very little of the theory of $C^*$-algebras. All we need is contained in \cite[Chapter 1,2,3,4]{Dixmier} and \cite{LanceBook}.
\paragraph*{Acknowledgments}
We thank N. Higson and G. Skandalis for their valuable remarks and suggestions.
This research was supported in part by the ANR project OpART (ANR-
23-CE40-0016).     

\section{Abstract pseudo-differential calculus and maximal hypoellipticity}\label{sec:Type I}
    We follow the following notational conventions: 
    \begin{itemize}
        \item    If $\pi$ is a representation of a $C^*$-algebra $A$ on a Hilbert space, then we denote the Hilbert space by $L^2(\pi)$.
        \item If $V,W$ are topological vector spaces, then $\cL(V,W)$ denotes the space of continuous linear maps from $V$ to $W$. If $V=W$, then $\cL(V):=\cL(V,V)$. 
        \item If $H$ is a Hilbert space, then $\cK(H)$ denotes the $C^*$-algebra of compact operators on $H$.
        \item Sesquilinear forms $\langle \xi,\eta\rangle$ are linear in $\eta$ and anti-linear in $\xi$.
    \end{itemize}
	\begin{prop}[Gelfand]\label{prop:simple_arg} Let $A$ be a unital $C^*$-algebra. 
			  An element $a\in A$ is left invertible if and only if $\pi(a):L^2(\pi)\to L^2(\pi)$ is injective for every irreducible representation $\pi$ of $A$.
	\end{prop}
	\begin{proof}
		The implication $\implies$ is clear. For the other implication, let $I\subseteq A$ be the closed left-ideal generated by $a$. If $a$ isn't left-invertible, then $I\neq A$.
		So, by \cite[Theorem 2.9.5]{Dixmier}, there exists a pure state $\phi:A\to \C$ such that $I\subseteq \{x\in A:\phi(x^*x)=0\}$.
		Hence, if $\pi$ is the irreducible representation obtained from $\phi$ by the GNS construction and $\xi$ is the distinguished vector, then $\pi(a)\xi=0$.
	\end{proof}
    \begin{prop}\label{prop:path_inv}
                If $(a_t)_{t\in [0,1]}$ is a continuous path in a unital $C^*$-algebra $A$ such that $a_0$ is invertible and $a_t$ are left-invertible for all $t\in [0,1]$, then $a_t$ are invertible for all $t\in [0,1]$.
    \end{prop}
    \begin{proof}
        Left invertibility of $a_t$ is equivalent to invertibility of $a_t^*a_t$. This can be seen by embedding $A$ in $\cL(H)$ for some Hilbert space $H$. 
       The element $a_t (a_t^*a_t)^{-1}a_t^*$ is a projection which depends continuously on $t$, and is equal to $1$ at $t=0$ because $a_0$ is invertible.
      So, it is equal to $1$ for all $t\in [0,1]$. Hence, $a_t$ is right-invertible whose inverse is $(a_t^*a_t)^{-1}a_t^*$.
    \end{proof}

	\begin{dfn}[Kaplansky, Glimm and Dixmier] A $C^*$-algebra $A$ is called \begin{itemize}
			\item Liminal if for every irreducible representation  $\pi$ of $A$, $\cK(L^2(\pi))=\pi(A)$.
			\item Type I if for every irreducible representation $\pi$ of $A$, $\cK(L^2(\pi))\subseteq \pi(A)$.
		\end{itemize}
	\end{dfn}
	We refer the reader to \cite[Chapters 4 and 9]{Dixmier} for a detailed discussion of liminal and $C^*$-algebras of Type I.

    We remark that if $A$ is a $C^*$-algebra and $\pi$ an irreducible representation such that $K(L^2(\pi))\cap \pi(A)\neq \{0\}$, then $K(L^2(\pi))\subseteq \pi(A)$, see \cite[Corollary 4.1.10]{Dixmier}.

    Let $A$ be a unital $C^*$-algbera of Type I.
    	By \cite[Proposition 4.3.3 and 4.3.4]{Dixmier}, there exists an ordinal $\alpha$ and a family of closed two-sided ideals $(I_{\rho})_{0\leq \rho\leq \alpha}$ of $A$ such that
		\begin{itemize}
			\item   $I_0=0$ and $I_\alpha=A$
			\item   If $\rho\leq \alpha$ is a limit ordinal, then $I_\rho=\overline{\bigcup_{\rho'<\rho}I_{\rho'}}$
			\item   If $\rho<\alpha$, then $I_\rho\subseteq I_{\rho+1}$ and $I_{\rho+1}/I_{\rho}$ is a non-zero liminal $C^*$-algebra.
		\end{itemize}
        \begin{rem}\label{rem:}
            The $C^*$-algebras that we will use in \MyCref{sec:Helffer_Nourrigat}, i.e., $C^*(G\rtimes\Rpt)$, the ordinal $\alpha$ is actually finite.
        \end{rem}

        Let $a\in A$.
        We define $\mathrm{ord}(a)<\alpha$ to be the smallest ordinal $\beta<\alpha$ such that $a$ is invertible in $A/I_{\beta}$. 
        If $a$ isn't invertible in any $A/I_{\beta}$, then $\mathrm{ord}(a):=\alpha$.
		Here, we follow the convention that $0$ isn't invertible in the zero ring, so nothing is invertible in $A/I_{\alpha}=0$.
        
			\begin{prop}\label{lem:smallest_ordinal_invertibility}
			  If $a\in A$, then $\mathrm{ord}(a)$ is equal to $0$ or it is a successor ordinal. Furthermore, if $\mathrm{ord}(a)\neq 0$, then for any irreducible representation of $A/I_{\mathrm{ord}(a)-1}$, $\pi(a):L^2(\pi)\to L^2(\pi)$ is a Fredholm operator, where $\mathrm{ord}(a)-1$ is the predecessor of $\mathrm{ord}(a)$.
		\end{prop}
		\begin{proof}
            Since $A$ is unital, $\alpha$ is a successor ordinal. 
            Let $\alpha-1$ be its predecessor.
		    So, $A/I_{\alpha-1}$ is a unital liminal $C^*$-algebra.
		    Hence, irreducible representations of $A/I_{\alpha-1}$ are finite dimensional.
            This prove the lemma if $\mathrm{ord}(a)=\alpha$.
            We now suppose that $\mathrm{ord}(a)<\alpha$.
            By hypothesis, there exist $b\in A$ and $c\in I_{\mathrm{ord}(a)}$ such that $ba=1+c$.
					If $\mathrm{ord}(a)$ is a limit ordinal, then there exists an ordinal $\beta<\mathrm{ord}(a)$ and $c'\in I_{\beta}$ such that $\norm{c-c'}<1$.
					Let $x=(1+c-c')^{-1}\in A$. So,
							$xba=1+xc'$.
					Since $xc'\in I_{\beta}$, it follows that $a$ is left invertible in $A/I_{\beta}$. By a similar argument, we deduce that $a$ is right invertible in $A/I_{\beta'}$ for some $\beta'<\mathrm{ord}(a)$.
					Hence, $b$ is invertible in $A/I_{\max(\beta,\beta')}$. We have thus obtained a contradiction to the minimality of $\mathrm{ord}(a)$.
               So, $\mathrm{ord}(a)$ is a successor ordinal.     
              Let $\pi$ be an irreducible representation of $A/I_{\mathrm{ord}(a)-1}$. There are two cases to consider
					\begin{itemize}
						\item   If $\pi$ vanishes on $I_{\mathrm{ord}(a)}/I_{\mathrm{ord}(a)-1}$, then $\pi$ descends to an irreducible representation of $A/I_{\mathrm{ord}(a)}$. Therefore, $\pi(a)$ is invertible and hence Fredholm.
						\item   If $\pi$ doesn't vanish on $I_{\mathrm{ord}(a)}/I_{\mathrm{ord}(a)-1}$, then the restriction of $\pi$ to $I_{\mathrm{ord}(a)}/I_{\mathrm{ord}(a)-1}$ is an irreducible representation of $I_{\mathrm{ord}(a)}/I_{\mathrm{ord}(a)-1}$, see \cite[Lemma 2.11.3]{Dixmier}.
							Since $I_{\mathrm{ord}(a)}/I_{\mathrm{ord}(a)-1}$ is liminal, we deduce that $\pi(I_{\mathrm{ord}(a)}/I_{\mathrm{ord}(a)-1})$ consists of compact operators.
							Invertibility of $a$ in $A/I_{\mathrm{ord}(a)}$ implies that $a$ is invertible in $A/I_{\mathrm{ord}(a)-1}$ modulo $I_{\mathrm{ord}(a)}/I_{\mathrm{ord}(a)-1}$.
							Therefore, $\pi(a)$ is invertible modulo compact operators.
							By Atkinson's theorem, $\pi(a)$ is Fredholm.\qedhere
					\end{itemize}
		\end{proof}
        \begin{cor}\label{thm:type_I_inv}
                Let $A$ be a unital $C^*$-algebra of Type I. An element $a\in A$ is invertible if and only if for every irreducible representation $\pi$ of $A$, the implication
\[
\pi(a) \text{ is Fredholm } \;\Longrightarrow\; \pi(a) \text{ is invertible}
\]
holds.
        \end{cor}
        \begin{proof}
            The forward implication is trivial. Suppose $a\in A$ such that if $\pi(a)$ is Fredholm, then $\pi(a)$ is invertible.
            If $\mathrm{ord}(a)$ is a successor ordinal, then by \MyCref{lem:smallest_ordinal_invertibility} and \MyCref{prop:simple_arg}, $a$ is invertible in $A/I_{\mathrm{ord}(a)-1}$ which contradicts minimality of $\mathrm{ord}(a)$.
            Hence, $\mathrm{ord}(a)=0$ which means $a$ is invertible in $A$.
        \end{proof}



	%
	\paragraph{Abstract $C^*$-calculus:}
    Let $\nu\geq 1$, a constant which plays the role of the number of multi-parametrs in the calculus.
    We follow the following conventions:
    \begin{itemize}
        \item        If $k,l\in \R^\nu$, then we write $k\leq l$ if $k_i\leq l_i$ for all $i\in \bb{1,\nu}$. 
        \item     If $k\in \C^\nu$, then $\Re(k):=(\Re(k_1),\cdots,\Re(k_\nu))\in \R^\nu$.
    \end{itemize}
    \begin{dfn}
        A $C^*$-calculus (with $\nu$-parameters) is a family of $\C$-vector spaces $(\cA_{k})_{k\in \C^\nu}$ such that for each $k,l\in \C^\nu$, we have a product map $\cA_k\times \cA_l\to \cA_{k+l}$ which is bilinear and associative, i.e., $(ab)c=a(bc)$ if $a\in \cA_{k_1}$, $b\in \cA_{k_2}$, $c\in \cA_{k_3}$.
				There is a unit element $0\neq 1\in \cA_0$, i.e., 1a$=a1=a$ for all $a\in \cA_k$.
                We also have an involution map $a\in \cA_k\mapsto a^*\in \cA_{\bar{k}}$ which is anti-linear and satisfies $1^*=1$ and $(ab)^*=b^*a^*$ for all $a\in \cA_{k_1}$, $b\in \cA_{k_2}$.
                We also have a $C^*$-semi-norm $\norm{\cdot}_{\overline{\cA_0}}$ on $\cA_0$, i.e., a semi-norm such that $\norm{ab}_{\overline{\cA_0}}\leq \norm{a}_{\overline{\cA_0}}\norm{b}_{\overline{\cA_0}}$ and $\norm{a^*a}_{\overline{\cA_0}}=\norm{a}^2_{\overline{\cA_0}}$ for all $a,b\in \cA_0$.
                We denote by $\overline{\cA_0}$ the Hausdorff completion of $\cA_0$.  
    \end{dfn}
    I don't suppose that I have an addition law of elements of $\cA_k$ and $\cA_l$ if $k\neq l$.
    For our applications, there is no need to suppose that. 
    Also, the sum of classical pseudo-differential operators on a smooth manifold of order $k\in \C$ and $l\in \C$ is a classical pseudo-differential operator if and only if $k-l\in \Z$.
    In fact in the example considered in \MyCref{sec:Helffer_Nourrigat}, see \eqref{eqn:homog_cond_ek}, the sum of element in $\cA_{k_1}$ and $\cA_{k_2}$ won't be in $\cA_{l}$ for any $l\in \C^\nu$.
    \begin{dfn}\label{dfn:unitary_representation_abstract}
        A representation of a $C^*$-calculus $(\cA_k)_{k\in \C^\nu}$ consists of a non-zero $\C$-vector space $C^\infty(\pi)$ equipped with a positive definite sesquilinear form $\langle \cdot,\cdot\rangle:C^\infty(\pi)\times C^\infty(\pi)\to \C$, 
        and a family of linear maps $\pi_k:\cA_k\to \mathrm{End}(C^\infty(\pi))$ for every $k\in \C^\nu$ such that the following hold:
		\begin{enumerate}
			\item   If $k,l\in \C^\nu$, $a\in\cA_k$, $b\in \cA_l$, $\xi,\eta\in C^\infty(\pi)$, then
                    \begin{equation*}\begin{aligned}
                     \pi_0(1)=\mathrm{Id}_{\C^\infty(\pi)},\quad  \pi_{k}(a)\pi_{l}(b)=\pi_{k+l}(ab),\quad   \langle  \pi_k(a)\xi,\eta\rangle=\langle \xi,\pi_{\bar{k}}(a^*)\eta\rangle.
                    \end{aligned}\end{equation*}
                
            \item If $k\in \C^\nu$ such that $\Re(k)\leq 0$ and $a\in \cA_{k}$, then $\pi_{k}(a)$ extends to a bounded operator $L^2(\pi)\to L^2(\pi)$, where $L^2(\pi)$ is the Hilbert space completion of $C^\infty(\pi)$ with respect to the norm $\norm{\xi}_{L^2(\pi)}^2:=\langle \xi,\xi\rangle$.
                    We also suppose that if $a\in \cA_0$, then 
                        \begin{equation}\label{eqn:bound_norm_rep}\begin{aligned}
                            \norm{\pi_0(a)}_{\cL(L^2(\pi))}\leq \norm{a}_{\overline{\cA_0}}.
                        \end{aligned}\end{equation}
                    So, $\pi_0$ extends to a representation of $\overline{\cA_0}$.
			\item The space $C^\infty(\pi)$ equipped with the topology generated by the family of semi-norms $\norm{\pi_k(a)\xi}_{L^2(\pi)}$ for $k\in \C^\nu$ and $a\in \cA_k$ is complete.
                    \end{enumerate}
            The representation is called irreducible if it is irreducible as a representation of the $C^*$-algebra $\overline{\cA_0}$, i.e., if $L\subseteq L^2(\pi)$ is a closed subspace which is invariant by $\pi_0(\cA_0)$, then $L=0$ or $L=L^2(\pi)$.
    \end{dfn}
            We denote by $C^{-\infty}(\pi)$ the space of continuous anti-linear functionals on $C^\infty(\pi)$. 
        The action of $\xi\in C^{-\infty}(\pi)$ on $\eta\in C^\infty(\pi)$ is denoted by $\langle \eta,\xi\rangle$.
        The space $C^{-\infty}(\pi)$ is equipped with the topology generated by the semi-norms $\xi\mapsto |\langle \xi,\eta\rangle|$ for every $\eta\in C^\infty(\pi)$.
		We have obvious continuous linear inclusions 
            \begin{equation*}\begin{aligned}
                C^\infty(\pi)\subseteq L^2(\pi)\subseteq C^{-\infty}(\pi).
            \end{aligned}\end{equation*}
		For each $k\in \C^\nu$, we extend $\pi_k$ to a linear map $\pi_k:\cA_k\to \cL(C^{-\infty}(\pi))$ by the formula 
            \begin{equation*}\begin{aligned}
                \langle \eta,\pi_k(a)\xi\rangle:=\langle \pi_{\bar{k}}(a^*)\eta,\xi \rangle,\quad \forall a\in \cA_k,\xi\in C^{-\infty}(\pi),\eta\in C^{\infty}(\pi).        
            \end{aligned}\end{equation*}

            There are various ways to define Sobolev spaces associated to the representation $\pi$.
            We will use here the most restrictive one. Our main theorem implies that they coincide with the least restrictive one, see \eqref{eqn:HkPi_alternate}.
		Let $H^k(\pi)$ be the set of $\xi\in C^{-\infty}(\pi)$ such that 
			there exists a sequence $(\xi_n)_{n\in \N}\subseteq C^\infty(\pi)$ such that $\xi_n\to \xi$ in $C^{-\infty}(\pi)$, and for any $l\in \C^\nu$ and $a \in\cA_l$ such that $\Re(l)\leq k$, one has $\pi_l(a)\xi\in L^2(\pi)$, and $\pi_l(a)\xi_n\to \pi_l(a)\xi$ in $L^2(\pi)$.
			We equip $H^k(\pi)$ with the topology generated by the family of semi-norms $\xi\mapsto \norm{\pi_l(a)\xi}_{L^2(\pi)}$ for all $l\in \C^\nu$ and $a\in \cA_l$ such that $\Re(l)\leq k$.
            Clearly, 
                \begin{equation}\label{eqn:inclusion_sobolev}\begin{aligned}
                    H^0(M)=L^2(\pi),\quad H^l(\pi)\subseteq H^{k}(\pi),\quad \text{if } k\leq l,
                \end{aligned}\end{equation}
            and if $k\in \C^\nu$, $l\in \R^\nu$ and $a\in \cA_k$, then 
                \begin{equation*}\begin{aligned}
                    \pi_k(a)(H^l(\pi))\subseteq H^{l-\Re(k)}(\pi)\text{, and the map }H^l(\pi)\xrightarrow{\pi_k(a)}  H^{l-\Re(k)}(\pi)\text{ is continuous.}
                \end{aligned}\end{equation*}
             The spaces $C^\infty(\pi)$, $C^{-\infty}(\pi)$, $H^k(\pi)$ are called the space of smooth vectors, distribution vectors and Sobolev spaces associated to the representation $\pi$ respectively.
            \begin{prop}One has
                \begin{equation}\label{eqn:Cinfinty_intersection}\begin{aligned}
                    C^\infty(\pi)=\bigcap_{k\in \R^\nu_+}H^k(\pi).
                \end{aligned}\end{equation}
            \end{prop}
            \begin{proof}
                Let $\xi \in \bigcap_{k\in \R^\nu_+}H^k(\pi)$.
                We will write $\xi$ as the limit of a net in $C^\infty(\pi)$.
                We will have to use nets because  $C^\infty(\pi)$ isn't necessarily metrizable.
                The index set of our net is the set of all finite subsets of $\bigsqcup_{k\in \C^\nu}\cA_k$ ordered by inclusion.
                For each finite subset $F\subseteq \bigsqcup_{k\in \C^\nu}\cA_k$, let $k\in \R_+^\nu$ such that if $F\cap \cA_l\neq \emptyset$, then  $\Re(l)\leq k$.
                Since $\xi\in H^{k}(\pi)$, let $(\xi_n)_{n\in \N}$ be as in the definition of $H^k(\pi)$.
                Let $\xi_F$ be an element in the sequence $\xi_n$ such that $\norm{\pi_l(a)\xi_F-\pi_l(a)\xi}_{L^2(\pi)}\leq \frac{1}{|F|}$ for all $a\in F$.
                The net $\xi_F$ is Cauchy in $C^\infty(\pi)$, so it has a limit in $C^\infty(\pi)$ which must be $\xi$.
            \end{proof}

        The next theorem is the main theorem of this article. 
        In the following theorem, let $K$ be the set of $k\in \C^\nu$ such that $k_i\neq 0$ for at most one $i\in\{1,\cdots,\nu\}$ and $\Re(k_i),\im(k_i)\in [-1,1]$.
            \begin{rem}\label{rem:K_smaller}
            In the following theorem, one can replace $[-1,1]$ in the definition of $K$ by $[-C,C]$ for any $C\in \Rpt$.
    \end{rem}
    \begin{theorem}\label{thm:simple_arg_sobolev}
        Let $(\cA)_{k\in \C^\nu}$ be a $C^*$-calculus, $\Sigma(\cA)$ a set of representations of $(\cA)_{k\in \C^\nu}$ such that the following hold:
        \begin{enumerate}
            \item The $C^*$-algebra $\overline{\cA_0}$ is of Type I.
            \item Any irreducible representation of $\overline{\cA_0}$ is unitary equivalent to $\pi_0$ for some $\pi\in \Sigma(\cA)$.
            \item We can find a family of elements $v_k\in \cA_k$ for $k\in  K$ such that: \begin{enumerate}
				\item The element $v_0$ is invertible in $\overline{\cA_0}$.
				\item For any $a\in \cA_0$, the map 
                    \begin{equation}\label{eqn:continuity_hypothesis}\begin{aligned}
                        k\in K \mapsto v_{k}av_{-k}\in \overline{\cA_0}
                    \end{aligned}\end{equation}
				is continuous, where $\overline{\cA_0}$ is equipped with the norm topology from $\norm{\cdot}_{\overline{\cA_0}}$.
        \end{enumerate}
        		\end{enumerate}

		Then following hold:
		\begin{enumerate}
            \item\label{thm:simple_arg_sobolev:3} Let $k\in\C^\nu$ and $a\in \cA_k$. If for any $\pi\in \Sigma(\cA)$, $C^\infty(\pi)\xrightarrow{\pi_k(a)} C^\infty(\pi)$ and $C^\infty(\pi)\xrightarrow{\pi_{\bar{k}}(a^*)} C^\infty(\pi)$ are injective, then for every representation $\pi$ of $(\cA)_{k\in \C^\nu}$ (not necessarily irreducible and not necessarily in $\Sigma(\cA)$), the maps
                                \begin{equation}\label{eqn:skqdjfmsqmjdfjsmjmfqsfd}\begin{aligned}
                                                                 &\pi_k(a):C^{-\infty}(\pi)\to C^{-\infty}(\pi)\\
&\pi_k(a):H^l(\pi)\to H^{l-\Re(k)}(\pi)&&\forall l\in \R^\nu\\
&\pi_k(a):C^\infty(\pi)\to C^\infty(\pi)
                                \end{aligned}\end{equation}
					are topological isomorphism. Furthermore, for any $k\in \C^\nu$ there exists $a\in \cA_k$ which satisfies the above.
 \item For any representation $\pi$ of $(\cA)_{k\in \C^\nu}$, $H^k(\pi)$ is a Hilbertian space, and \begin{equation}\label{eqn:HkPi_alternate}\begin{aligned}
                        		H^k(\pi)=\{\xi\in C^{-\infty}(\pi):\pi_k(a)\xi\in L^2(\pi),\ \forall a\in \cA_k\}		
                    \end{aligned}\end{equation} and $C^{-\infty}(\pi)=\bigcup_{k\in \R^\nu}H^k(\pi)$.
				%
		
			\item\label{thm:simple_arg_sobolev:inv} If $a\in \cA_0$, then $a$ is left-invertible in $\overline{\cA_0}$ if and only if for all $\pi\in \Sigma(\cA)$, $C^\infty(\pi)\xrightarrow{\pi_0(a)} C^\infty(\pi)$ is injective
		\end{enumerate}
	\end{theorem}
    In the case where the map $\cA_0\to \overline{\cA_0}$ isn't injective then, when we say that $a\in \cA_0$ is invertible in $\overline{\cA_0}$, we mean that its image in $\overline{\cA_0}$ is invertible.
    \begin{exs}\label{ex:}
        \begin{enumerate}
            \item    
                Let $M$ be a smooth compact manifold, $\cA_k$ be the space of classical pseudo-differential operators of order $k$ on $M$, and 
                $\norm{\cdot}_{\overline{\cA_0}}$ is the operator norm on $L^2(M)$, $\Sigma(\cA)=T^*M\backslash 0\sqcup \{\heartsuit\}$, where $\heartsuit$ is an auxillary point.
                For any $(x,\xi)\in T^*M\backslash 0$, we have an obvious irreducible representation of $(\cA)_{k\in \C}$ given by the principal symbol.
				Let $C^\infty(\heartsuit)=C^\infty(M,\omegahalf)$, $L^2(\heartsuit)=L^2(M)$, and $\heartsuit_k$ be the usual action of pseudo-differential operators on functions.
				In this setting, \MyCref{thm:simple_arg_sobolev}[3] says that if $P$ is an elliptic classical pseudo-differential $P$ of order $k$ such that $P$ and $P^*$ are injective on $C^\infty(M,\omegahalf)$, then 
                    \begin{equation*}\begin{aligned}
                        C^{-\infty}(M,\omegahalf)&\xrightarrow{P} C^{-\infty}(M,\omegahalf)&& \\
                        	 H^{k+l}(M) &\xrightarrow{P}H^l(M),&&\quad \forall l\in \R\\
                            C^{\infty}(M,\omegahalf)&\xrightarrow{P} C^{\infty}(M,\omegahalf)&&
                            \end{aligned}\end{equation*}
                are topological isomorphisms.
				This implies the main regularity theorem of elliptic classical pseudo-differential operators.
                To see this,
                if $Q$ is an elliptic operator, then one can apply the above to $P=Q^*Q+R$, where $R$ is a positive smoothing operator which is injective on $C^{\infty}(M,\omegahalf)$.
\item 
    Let $(\cA)_{k\in \C^{\nu_1}}$ and $(\cB_{k})_{k\in \C^{\nu_2}}$ be two $C^*$-calculi.
    We can define a $C^*$-calculus $\cC$ by $\cC_{(k,l)}=\cA_k\otimes \cB_l$ where $(k,l)\in \C^{\nu_1+\nu_2}$.
    We take $\Sigma(\cC)=\Sigma(\cA)\times \Sigma(\cB)$.
    If $(\pi,\pi')\in \Sigma(\cC)$, then for any $(k,l)\in \R^{\nu_1+\nu_2}$, we take $H^{(k,l)}(\pi,\pi'):=H^k(\pi)\otimes H^l(\pi') $ the Hilbert tensor product.
    One has $H^{(k,l)}(\pi,\pi')\subseteq H^{(k',l')}(\pi,\pi')$ if $(k',l')\leq (k,l)$.
    One then takes 
        \begin{equation*}\begin{aligned}
            C^\infty(\pi,\pi'):=\bigcap_{(k,l)\in \R^{\nu_1+\nu_2}}H^{(k,l)}(\pi,\pi'), \quad C^{-\infty}(\pi,\pi'):=\bigcup_{(k,l)\in \R^{\nu_1+\nu_2}}H^{(k,l)}(\pi,\pi')
        \end{aligned}\end{equation*}
    The maps $(\pi,\pi')_{(k,l)}:\cC_{(k,l)}\to \cL(C^{-\infty}(\pi,\pi'))$ are obviously defined.
    The only non-trivial fact to check is that the tensor product of two $C^*$-algebras of Type \RNum{1} is again of Type \RNum{1}, and the spectrum of the tensor product is the product of the spectrums, see \cite{GuichardetTensorProductTypeI}.
    It is interesting to note that if one takes the tensor product of two classical pseudo-differential calculi on manifolds $M_1$ and $M_2$, one obtains the symbol map that appears in Rodino's calculus \cite{RodinoProductCalculus}.
        \end{enumerate}
    \end{exs}
	\begin{proof}[Proof of \Cref{thm:simple_arg_sobolev}]
            If $a\in \cA_0$, then we say that $a$ is invertible if $a$ is invertible in $\overline{\cA_0}$.
            Throughout the proof, whenever we consider an irreducible representation of $\overline{\cA_0}$, we will suppose that it is of the form $\pi_0$ for some $\pi\in \Sigma(\cA)$.
            We will also be careful in distinguishing between arbitrary representations $\pi$ of $(\cA)_{k\in \C^\nu}$ and representations in $\Sigma(\cA)$.
            
            By continuity of $v_kv_{-k}$, it follows that $v_{k}v_{-k}$ is invertible for $k$ close enough to $0$. 
            So, we can suppose that $v_kv_{-k}$ is invertible for all $k\in  K$ (otherwise one rescales all indices).
            Let $K_+:=K\cap [0,1]^\nu$ and $K_{\R}=K_+\cup -K_+=K\cap \R^\nu$.
            \begin{lem}\label{lem:step14}
                 If $k\in K_\R$, $l\in \R^\nu$ and $\pi$ is a representation of $(\cA)_{k\in \C^\nu}$, then the maps 
                     \begin{equation*}\begin{aligned}
                        C^{-\infty}(\pi)\xrightarrow{\pi_k(v_k)} C^{-\infty}(\pi),\quad
                        H^l(\pi)\xrightarrow{\pi_{k}(v_{k})} H^{l-k}(\pi) ,\quad
                        C^\infty(\pi)\xrightarrow{\pi_k(v_k)} C^\infty(\pi)
                     \end{aligned}\end{equation*}
                    are topological isomorphisms. The same result holds with $v_k$ replaced by $v_k^*$.
            \end{lem}
            \begin{proof}
            The proof relies on a bootstrapping argument.
                Let $k\in K_\R$. 
                 Since $v_kv_{-k}$ is invertible, the map $L^2(\pi)\xrightarrow{\pi_0(v_kv_{-k})}L^2(\pi)$ is invertible. This map is the composition 
                    \begin{align}
                        L^2(\pi)\xrightarrow{\pi_{-k}(v_{-k})}H^k(\pi)\xrightarrow{\pi_{k}(v_k)}   L^2(\pi)\label{eqn:qsdifjkqsjdkmlfqsdf}
                    \end{align}
                Injectivity of \eqref{eqn:qsdifjkqsjdkmlfqsdf} implies that $ L^2(\pi)\xrightarrow{\pi_{-k}(v_{-k})}H^k(\pi)$ is injective.
                 If $k\geq 0$, then by replacing $k$ with $-k$ and by the inclusion $H^k(\pi)\subseteq L^2(\pi)$, it follows that $H^k(\pi)\xrightarrow{\pi_{k}(v_k)}   L^2(\pi)$ is injective.
                The composition of two injective maps is a bijection implies that each is a bijection.
                So each map in \eqref{eqn:qsdifjkqsjdkmlfqsdf} is bijective (if $k\in K_+$).
                They are  topological isomorphisms because their inverses are $H^k(\pi)\xrightarrow{\pi_{k}(v_{k})}L^2(\pi)\xrightarrow{\pi_{0}(v_{k}v_{-k})^{-1}}L^2(\pi) $ and  $L^2(\pi)\xrightarrow{\pi_{0}(v_{k}v_{-k})^{-1}}L^2(\pi)\xrightarrow{\pi_{-k}(v_{-k})}H^{-k}(\pi)$ which are continous.
                We have thus proved that if $k\in K_+$, then the maps $L^2(\pi)\xrightarrow{\pi_{-k}(v_{-k})}H^k(\pi)$ and $H^k(\pi)\xrightarrow{\pi_{k}(v_k)}   L^2(\pi)$ are topological isomorphisms.

                Let $k,k'\in K_+$.
                Consider $v_{k}v_{k'}v_{-k'}v_{-k}\in \cA_0$. If $k'$ is fixed, then these operators depends continuously on $k$ by \eqref{eqn:continuity_hypothesis}.
                At $k=0$, they are invertible because $v_0$ and $v_{k'}v_{-k'}$ are invertible.
               For any representation $\pi$ of $(\cA)_{k\in \C^\nu}$, $L^2(\pi)\xrightarrow{\pi_0(v_{k}v_{k'}v_{-k'}v_{-k})}L^2(\pi)$ is the composition 
                   \begin{equation}\label{eqn:sdkfkjlskdjflsd}\begin{aligned}
                        L^2(\pi)\xrightarrow{\pi_{-k}(v_{-k})}H^k(\pi) \xrightarrow{\pi_{-k'}(v_{-k'})} H^{k+k'}(\pi)\xrightarrow{\pi_{k'}(v_{k'})}H^k(\pi)\xrightarrow{\pi_k(v_k)}L^2(\pi)      
                   \end{aligned}\end{equation}
               Each of these applications is injective by what we proved above, and the inclusions $H^k(\pi)\subseteq L^2(\pi)$ and $H^{k+k'}(\pi)\subseteq H^{k'}(\pi)$.
               By \Cref{prop:simple_arg} and our assumption that any irreducible representation is equivalent to $\pi_0$ for some $\pi\in \Sigma(\cA)$, $v_{k}v_{k'}v_{-k'}v_{-k}$ is left-invertible.
               So, by \Cref{prop:path_inv}, $v_{k}v_{k'}v_{-k'}v_{-k}$ is invertible.
               Hence, \eqref{eqn:sdkfkjlskdjflsd} is invertible.
               The composition of injective maps is bijective implies that each is bijective.
               And by the same argument as above, each is a topological isomorphism.
               By repeating the same argument with $v_{k_1}\cdots v_{k_n}v_{-k_n}\cdots v_{-k_1}$, it follows that $H^l(\pi)\xrightarrow{\pi_{-k}(v_{-k})} H^{l+k}(\pi)$ and  $H^{l+k}(\pi)\xrightarrow{\pi_{k}(v_{k})}H^l(\pi)$ are topological isomorphisms
               whenever $k\in K_+$, and $l\in \R_+^\nu$.

               Taking the limit as $l\to \infty$ and using \eqref{eqn:Cinfinty_intersection}, it follows that for any $k\in K_\R$, $C^\infty(\pi)\xrightarrow{\pi_k(v_k)} C^\infty(\pi)$ is a topological isomorphism.
               The properties satisfied by $v_k$ are also satisfied by $v_k^*$.
               So, by the same argument $C^\infty(\pi)\xrightarrow{\pi_k(v_k^*)} C^\infty(\pi)$ is a topological isomorphism.  
               By duality, $C^{-\infty}(\pi)\xrightarrow{\pi_k(v_k)} C^{-\infty}(\pi)$ and $C^{-\infty}(\pi)\xrightarrow{\pi_k(v_k^*)} C^{-\infty}(\pi)$ are topological isomorphisms for any $k\in K_\R$.

               Let $k\in K_\R$ and $l\in \R^\nu$. We need to show that $H^l(\pi)\xrightarrow{\pi_{k}(v_{k})} H^{l-k}(\pi)$ is a topological isomorphism.
               Let $k_1,\cdots,k_n,k_1',\cdots,k_m'\in K_\R$ such that $k_1+\cdots+k_n=-l$ and $k_1'+\cdots k_m'=l-k$. Consider $w=v_{k_1'}\cdots v_{k_m'}v_kv_{k_1}\cdots v_{k_n}\in \cA_0$.
                The map $L^2(\pi)\xrightarrow{\pi_0(w)} L^2(\pi)$ is equal to 
               \begin{equation}\label{eqn:sqdiofjmkdsqjmfjqmsdjf}\begin{aligned}
                            L^2(\pi)\xrightarrow{\pi_{k_n}(v_{k_n})}H^{-k_n}(\pi) \cdots \to H^{l}(\pi)\xrightarrow{\pi_k(v_k)}H^{l-k}(\pi) \to\cdots  \xrightarrow{\pi_{k_1'}(v_{k_1'})}L^2(\pi)
                    \end{aligned}\end{equation}
               Each of these maps is injective by injectivity on $C^{-\infty}(\pi)$ that we proved
               Hence, $w$ is left-invertible by \Cref{prop:simple_arg}.
               By the same argument on the adjoint, $w$ is invertible.
               So $L^2(\pi)\xrightarrow{\pi_0(w)} L^2(\pi)$ is bijective which implies that each map in \eqref{eqn:sqdiofjmkdsqjmfjqmsdjf} is bijective.
               By a similar argument, we show that each is a topological isomorphism.
            \end{proof}
            \begin{lem}\label{lem:step15}
                For any $a\in \cA_0$ and $k\in K_+$, $\mathrm{ord}(v_kav_{-k})\leq \mathrm{ord}(a)$.
            \end{lem}
            \begin{proof}
                Without loss of generality, we can suppose that $\mathrm{ord}(a)<\alpha$.
                We need to show that $v_kav_{-k}$ is invertible in $\overline{\cA_0}/I_{\mathrm{ord}(a)}$.
                By our assumption of continuity, the set of $k\in K_+$ such that $v_kav_{-k}$ is invertible in $\overline{\cA_0}/I_{\mathrm{ord}(a)}$ is open.
                It is non-empty because $v_0$ is invertible in $\overline{\cA_0}$.
                We will show that it is closed.
                Let $k$ be a point in the closure.
                Suppose $v_kav_{-k}$ isn't invertible in $\overline{\cA_0}/I_{\mathrm{ord}(a)}$.
                So, $ \mathrm{ord}(a)<\mathrm{ord}(v_kav_{-k})$.
                Let $\pi\in \Sigma(\cA)$ be an irreducible representation  $\overline{\cA_0}/I_{\mathrm{ord}(v_kav_{-k})-1}$.
                The operator $L^2(\pi)\xrightarrow{\pi_0(v_kav_{-k})}L^2(\pi)$ is Fredholm operator (by \MyCref{lem:smallest_ordinal_invertibility}) which is a limit of invertible operators (because $k$ is in the closure and $\mathrm{ord}(a)\leq \mathrm{ord}(v_kav_{-k})-1$).
                So, it has Fredholm index $0$.
                The map  $L^2(\pi)\xrightarrow{\pi_0(v_kav_{-k})}L^2(\pi)$ is the composition of the maps 
                    \begin{equation*}\begin{aligned}
                            L^2(\pi)\xrightarrow{\pi_{-k}(v_{-k})}H^k(\pi)\xrightarrow{\pi_0(a)}H^k(\pi)\xrightarrow{\pi_k(v_k)}L^2(\pi)        
                    \end{aligned}\end{equation*}
                The outer maps are injective by \Cref{lem:step14}. The inner one is injective because $L^2(\pi)\xrightarrow{\pi_0(a)}L^2(\pi)$ is invertible ($\mathrm{ord}(a)\leq \mathrm{ord}(v_kav_{-k})-1$ and $H^k(\pi)\subseteq L^2(\pi)$).
                So, the map $L^2(\pi)\xrightarrow{\pi_0(v_kav_{-k})}L^2(\pi)$ is injective Fredholm of index 0. Hence, it is invertible.
                So, by \Cref{prop:simple_arg}, $v_kav_{-k}$ is invertible in $\overline{\cA_0}/I_{\mathrm{ord}(v_kav_{-k})-1}$ which contradicts the minimality of $\mathrm{ord}(v_kav_{-k})$.
            \end{proof}

            \begin{lem}\label{lem:step2}
                 If $a\in \cA_0$ such that for any $\pi\in \Sigma(\cA)$, $C^\infty(\pi)\xrightarrow{\pi_0(a)} C^\infty(\pi)$ is injective, then $a$ is left-invertible in $\overline{\cA_0}$.
            \end{lem}
            \begin{proof}
				By replacing $a$ by $a^*a$, we can suppose that $a$ is self-adjoint.
                We will show that $\mathrm{ord}(a)=0$.
                Suppose it isn't.
                We will obtain a contradiction by showing that $a$ is invertible in $\overline{\cA}_0/I_{\mathrm{ord}(a)-1}$.
                Let $\pi\in \Sigma(\cA)$ be an irreducible representation of $\overline{\cA_0}/I_{\mathrm{ord}(a)-1}$.
                Since $a$ is self-adjoint, $L^2(\pi)\xrightarrow{ \pi_0(a)}L^2(\pi)$ is a Fredholm operator of index $0$.

                We now consider the elements $v_{k_1}\cdots v_{k_n}a v_{-k_n}\cdots v_{-k_1}\in \cA_0$ where $k_1,\cdots,k_n\in K_+$ and $n\in \N$.
                By iterated application of \Cref{lem:step15}, we have 
                    \begin{equation*}\begin{aligned}
                        \mathrm{ord}(v_{k_1}\cdots v_{k_n}a v_{-k_n}\cdots v_{-k_1})\leq \mathrm{ord}(a)     
                    \end{aligned}\end{equation*}
                Therefore, $L^2(\pi)\xrightarrow{\pi_0(v_{k_1}\cdots v_{k_n}a v_{-k_n}\cdots v_{-k_1})} L^2(\pi)$ is a Fredholm operator. To see this, either $ \mathrm{ord}(v_{k_1}\cdots v_{k_n}a v_{-k_n}\cdots v_{-k_1})=\mathrm{ord}(a)$ in which case \Cref{lem:smallest_ordinal_invertibility} gives Fredholmness, 
                or $ \mathrm{ord}(v_{k_1}\cdots v_{k_n}a v_{-k_n}\cdots v_{-k_1})<\mathrm{ord}(a)$ in which case $L^2(\pi)\xrightarrow{\pi_0(v_{k_1}\cdots v_{k_n}a v_{-k_n}\cdots v_{-k_1})} L^2(\pi)$ is invertible.
                By iterated application of our continuity hypothesis \eqref{eqn:continuity_hypothesis} and the hypothesis $v_0$ is invertible, they all have the same Fredholm index as $L^2(\pi)\xrightarrow{ \pi_0(a)}L^2(\pi)$. So, their Fredholm index vanishes.
                
                Let $V=\ker(L^2(\pi)\xrightarrow{ \pi_0(a)}L^2(\pi))$.
                Since $H^k(\pi)\subseteq H^l(\pi)$ if $l\leq k$ and $V$ is finite dimensional, there exists $l\in \R_+^\nu$ such that 
					\begin{equation*}\begin{aligned}
						V\cap H^l(\pi)=\bigcap_{k\in \R^\nu_+}V\cap H^k(\pi)=V\cap C^\infty(\pi),
					\end{aligned}\end{equation*}
                where the second equality follows from \eqref{eqn:Cinfinty_intersection}.
				By our assumption, $\pi_0(a)$ is injective on $C^\infty(\pi)$, that is $V\cap C^\infty(\pi)=0$. 
				So, $V\cap H^l(\pi)=0$.
                Fix $k_1,\cdots,k_n\in K_+$ such that $k_1+\cdots+k_n=l$.
                The map $L^2(\pi)\xrightarrow{\pi_0(v_{k_1}\cdots v_{k_n}a v_{-k_n}\cdots v_{-k_1})} L^2(\pi)$ is equal to the composition
                \begin{equation}\label{eqn:sqdjfjlmsqdkjflmqsd}\begin{aligned}
                            L^2(\pi)\xrightarrow{\pi_{-k_1}(v_{-k_1})}H^{k_1}(\pi) \cdots H^{l}(\pi)\xrightarrow{\pi_0(a)}H^l(\pi) \cdots  \xrightarrow{\pi_{k_1}(v_{k_1})}L^2(\pi)
                    \end{aligned}\end{equation}
                The exterior maps are topological isomorphisms by \Cref{lem:step14}.
                The interior map is injective because $V\cap H^l(\pi)=0$.
                So, it follows that $L^2(\pi)\xrightarrow{\pi_0(v_{k_1}\cdots v_{k_n}a v_{-k_n}\cdots v_{-k_1})} L^2(\pi)$ is injective.
                Hence, it is bijective (because it is Fredholm of index $0$).
                So, each map in \eqref{eqn:sqdjfjlmsqdkjflmqsd} is bijective.
                So, $H^{l}(\pi)\xrightarrow{\pi_0(a)}H^l(\pi)$ is surjective.
                Therefore, $L^2(\pi)\xrightarrow{\pi_0(a)} L^2(\pi)$ has dense image because $C^\infty(\pi)\subseteq H^l(\pi)\subseteq L^2(\pi)$. A Fredholm operator of Fredholm index $0$, whose image is dense, is invertible. 
				Hence, $L^2(\pi)\xrightarrow{\pi_0(a)} L^2(\pi)$ is invertible.
                By \Cref{prop:simple_arg}, $a$ is invertible in $\overline{\cA}_0/I_{\mathrm{ord}(a)-1}$ which is a contradiction.
		 \end{proof}
             We now prove that if $k\in \C^\nu$, $a\in \cA_k$ satisfies $C^\infty(\pi)\xrightarrow{\pi_k(a)} C^\infty(\pi)$ and $C^\infty(\pi)\xrightarrow{\pi_{\bar{k}}(a^*)} C^\infty(\pi)$ are injective for any $\pi\in \Sigma(\cA)$, then \eqref{eqn:skqdjfmsqmjdfjsmjmfqsfd} are topological isomorphisms for any representation $\pi$ of $(\cA)_{k\in \C^\nu}$.
            We need to show that \eqref{eqn:skqdjfmsqmjdfjsmjmfqsfd} have continuous left and right inverses.
            So, by replacing $a$ with $a^*a$ and $aa^*$, we can suppose that $k\in \R^\nu$ and $a=a^*$.
            Let $l\in \R^\nu$.
            We fix any elements $k_1,\cdots,k_n,k_1',\cdots,k_m'\in K_\R$
            such that $k_1+\cdots+k_n=l-k$ and $k_1'+\cdots+k_m'=-l$.
			By \Cref{lem:step2} and \Cref{lem:step14}, $v_{k_1}\cdots v_{k_n}a v_{k_1'}\cdots v_{k_m'}\in \cA_0$ is left-invertible in $\overline{\cA_0}$ for any $l\in \R^\nu$.
			By the same argument applied to the adjoint, we deduce that it is invertible.
            Now, let $\pi$ be a representation of $(\cA)_{k\in \C^\nu}$.
			So, $L^2(\pi)\xrightarrow{\pi_0(v_{k_1}\cdots v_{k_n}a v_{k_1'}\cdots v_{k_m'})} L^2(\pi)$ is invertible.
            By factorizing this map and using \Cref{lem:step14}, it follows that $ H^{l}(\pi) \xrightarrow{\pi_{k}(a)}H^{l-k}(\pi)$ is a topological isomorphism.
            By taking the limit as $l\to \infty$ and using \eqref{eqn:Cinfinty_intersection},  $ C^\infty(\pi)\xrightarrow{\pi_{k}(a)}C^\infty(\pi)$ is a topological isomorphism.
            By duality, $ C^{-\infty}(\pi)\xrightarrow{\pi_{k}(a)}C^{-\infty}(\pi)$ is a topological isomorphism.

            We now show that for any $k\in \C^\nu$, then there exists $a\in \cA_k$ such that the maps \eqref{eqn:skqdjfmsqmjdfjsmjmfqsfd} are topological isomorphisms for any representation $\pi$ of $(\cA)_{k\in \C^\nu}$.
            By what we proved it suffices to have that $C^\infty(\pi)\xrightarrow{\pi_k(a)} C^\infty(\pi)$ and $C^\infty(\pi)\xrightarrow{\pi_k(a^*)} C^\infty(\pi)$ are injective for any $\pi\in \Sigma(\cA)$.
            We write $k$ as sum of $k_1+\cdots+k_n$ with $k_i\in K$, and take $a=v_{k_1}\cdots v_{k_n}$.
            Since by hypothesis (see the beginning of this proof, we also replaced $k$ with $-k$) $v_{-k}v_{k}$ is invertible in $\overline{\cA_0}$ for any $k\in K$, it follows that $\pi_k(v_k)$ is injective on $L^2(\pi)$ and in particular on $C^\infty(\pi)$ for any $\pi \in \Sigma(\cA)$.
            Invertibility of $v_{-k}v_{k}$ implies invertibility of its adjoint. So by the same argument $\pi_{\overline{k}}(v_k^*)$ is injective on $C^\infty(\pi)$.
                        We now show that $C^{-\infty}(\pi)=\bigcup_{k\in \R^\nu} H^k(\pi)$ for any representation $\pi$ of $(\cA)_{k\in \C^\nu}$. 
                        Let $\xi\in C^{-\infty}(\pi)$. By the definition of the topology on $C^{\infty}(\pi)$, there
                        exists $C>0$, and $a_1\in \cA_{k_1},\cdots,a_n\in \cA_{k_n}$ such that 
                            \begin{equation*}\begin{aligned}
                                |\langle \eta,\xi\rangle| \leq C\sum_{i=1}^n \norm{\pi_{k_i}(a_i)\eta}_{L^2(\pi)},\quad \forall\eta\in C^\infty(\pi)      
                            \end{aligned}\end{equation*}
                        Let $k\in \R^\nu$ such that $k_i\leq k$ for all $i$, and $w\in \cA_k$ such that \eqref{eqn:skqdjfmsqmjdfjsmjmfqsfd} are topological isomorphisms. So, $H^k(\pi)\xrightarrow{\pi_k(w)}L^2(\pi)$ is a topological isomorphism. 
                        Hence,
                        \begin{equation}\label{eqn:qjsidofjmosqjdifjiqsmdof}\begin{aligned}
                                |\langle \eta,\xi\rangle| \leq C'\norm{\pi_{k}(w)\eta}_{L^2(\pi)},\quad \forall\eta\in C^\infty(\pi)      
                            \end{aligned}\end{equation}
						The map $C^{-\infty}(\pi)\xrightarrow{\pi_{k}(w)} C^{-\infty}(\pi)$ is a topological isomorphism.
                        So, there exists $\xi'\in C^{-\infty}(\pi)$ such that $\pi_k(w)\xi'=\xi$.
                        By \eqref{eqn:qjsidofjmosqjdifjiqsmdof} together with surjectivity $C^\infty(\pi)\xrightarrow{\pi_k(w)} C^\infty(\pi)$, we deduce that $\xi'\in L^2(\pi)$.
                        So, $\xi\in H^{-k}(\pi)$.\qedhere

	\end{proof}

	\begin{ques}\label{ques:thm_simp_on_non_type_I}
		Does \Cref{thm:simple_arg_sobolev} hold without $\overline{\cA_0}$ being necessarily of Type \RNum{1}?
	\end{ques}
\section{Helffer-Nourrigat's theorem}\label{sec:Helffer_Nourrigat}
    The material in this section is mostly folklore (mostly due to Glowski, Goodman, Helffer, Knapp, Nourrigat, Stein, Taylor \cite{Glowacki2,HelfferRockland,GoodmanBoundedL2,TaylorBook,KnappStein1} and their collaborators) except for the proof of \MyCref{thm:Helffer_Nourrigat}.

        A graded nilpotent Lie algebra of depth $N\in \N$ is a graded vector space $\mathfrak{g}=\mathfrak{g}^1\oplus\cdots \oplus\mathfrak{g}^N$ equipped with a Lie algebra structure such that for every $i,j\in \{1,\cdots,N\}$ $$[\mathfrak{g}^i,\mathfrak{g}^j]\subseteq \begin{cases}\mathfrak{g}^{i+j}, &\text{if }i+j\leq N\\0&\text{if not}\end{cases}$$
        Since $\mathfrak{g}$ is nilpotent, the Baker-Campbell-Hausdorff (BCH) formula 
    \begin{equation}\label{eqn:bch}
        g\cdot h:=g+h-\frac{1}{2}[g,h]+\frac{1}{12}[g,[g,h]]-\frac{1}{12}[h,[h,g]]+\cdots
    \end{equation} 
    is a finite sum.   Our BCH formula differs by a sign from the one used in group theory, i.e., our formula is the usual formula for the opposite Lie algebra. 
        Hence, if $X_g$ is the \textit{right} invariant vector field on $\mathfrak{g}$ such that $X_g(0)=g$, then $[X_g,X_h]=X_{[g,h]}$.
    
        The space $\mathfrak{g}$ equipped with the product given by the BCH formula is a simply connected Lie group integrating $\mathfrak{g}$.
    We treat $\mathfrak{g}$ as both a Lie group and a Lie algebra through the article.
    Whenever we want to stress on the Lie group aspect, we will use the notation $G$ instead.
    The inverse of $g\in \mathfrak{g}$ will be denoted by $-g$ and $g^{-1}$ depending on context. 
    A Lebesgue measure on $\mathfrak{g}$ is also a Haar measure on $G$.
    We fix such a measure and denote it by $\odif{g}$.
    
    Let $ C^{-\infty}(\mathfrak{g})$ be the topological anti-linear dual of $C^\infty_c(\mathfrak{g})$, and $C^{-\infty}_c(\mathfrak{g})$ the subspace of compactly supported distributions.
    The action of a distribution $u$ on a function $f$ will be denoted by $ \langle f, u \rangle $.
    In other words, we will only use the inner product $\langle \phi,\psi\rangle=\int \bar{\phi}\psi$, where $\phi,\psi\in C^\infty_c(\mathfrak{g})$.


   \paragraph{Convolution:} If $u,v\in C^{-\infty}(\mathfrak{g})$, then we define the convolution and transpose using the group law 
       \begin{equation*}\begin{aligned}
            u\ast v(g)=\int u(gh^{-1})v(h)\odif{h},\quad u^t(g)=u(g^{-1}),\quad u^*=\overline{u^t}.        
       \end{aligned}\end{equation*}
    The convolution is well-defined if either $u$ or $v$ has compact support.
   The element $\delta$ denotes the Dirac delta distribution at $0$ which is the unit of the convolution product.
   
   
\paragraph{Right-invariant differential operators:} 
Let $D:C^\infty(\mathfrak{g})\to C^\infty(\mathfrak{g})$ be a right-invariant differential operator on $G$. 
We denote by $\delta_D\in C^\infty_c(\mathfrak{g})$ the unique distribution which satisfies $ \delta_D\ast f=D(f)$ for all $f\in C^\infty(\mathfrak{g})$.
   Let $D^*$ be the formal adjoint of $D$ which is also right-invariant. One can check that
          \begin{equation*}\begin{aligned}
 f\ast \delta_D=D^*(f),\quad \delta_{D_1}\ast \delta_{D_2}=\delta_{D_1D_2},\quad \delta_D^*=\delta_{D^*}.
    \end{aligned}\end{equation*}
The map $D\mapsto \delta_D$ is a $*$-algebra isomorphism between the $*$-algebra of right-invariant differential operators on $\mathfrak{g}$ and the $*$-subalgebra of $C^{-\infty}_c(\mathfrak{g})$ of distributions supported in $\{0\}$

\paragraph{Dilations:}
We define the dilation $$\alpha_\lambda:\mathfrak{g}\to \mathfrak{g},\quad  \alpha_\lambda\left(\sum_{i=1}^Ng_i\right)=\sum_{i=1}^N\lambda^ig_i,\quad \lambda\in \R_+.$$
The map $\alpha_\lambda$ is a Lie algebra and a Lie group homomorphism. The dilation action on distributions is defined by 
\begin{equation}\label{eqn:dil_dis}
    \alpha_\lambda: C^{-\infty}(\mathfrak{g})\to  C^{-\infty}(\mathfrak{g}),\quad 
    \langle \phi,\alpha_\lambda(u)\rangle=\langle \phi \circ \alpha_\lambda,u\rangle,\quad 
    \phi\in C^\infty_c(\mathfrak{g}).
\end{equation}
So, if $u\in C^\infty(\mathfrak{g})$, then 
    \begin{equation}\label{eqn:dilation_Q}\begin{aligned}
        \alpha_\lambda(u)=\lambda^{-Q} u\circ \alpha_{\lambda^{-1}}, \quad Q=\sum_{i=1}^Ni\dim(\mathfrak{g}^i).
    \end{aligned}\end{equation}
Hence, if $u\in C^\infty_c(\mathfrak{g})$, then  \begin{equation}\label{eqn:norm_dilations}\begin{aligned}
        \supp(\alpha_{\lambda}(u))\subseteq \alpha_\lambda(\supp(u)),\quad \norm{\alpha_{\lambda}(u)}_{L^1}=\norm{u}_{L^1},\quad \forall \lambda\in \R_+^\times.
    \end{aligned}\end{equation}
    It is useful to note that 
        \begin{equation}\label{eqn:limit_delta_0}\begin{aligned}
            \lim_{\lambda\to 0^+}\alpha_\lambda(u)=\langle 1,u\rangle \delta   ,\quad \forall u\in C^{-\infty}_c(\mathfrak{g}),    
        \end{aligned}\end{equation}
        where the limit is in $C^{-\infty}_c(\mathfrak{g})$.

        Notice that $P\in C^\infty(\mathfrak{g})$ is homogeneous as a distribution of degree $-n-Q$ if and only if  $P\circ \alpha_\lambda=\lambda^n P$.
   Furthermore, $P$ is necessarily a polynomial because of smoothness at $0$.

 \paragraph{Homogeneous differential operators:} A right-invariant differential operator $D$ is said to be homogeneous of degree $k$ if $\delta_D$ is homogeneous of degree $k$, i.e., if \begin{equation}
    D(\phi\circ \alpha_\lambda)=\lambda^kD(\phi)\circ\alpha_\lambda\quad  \forall \phi\in C^\infty_c(\mathfrak{g}),
   \end{equation} 
   We denote by $\DO^k(\mathfrak{g})$ the space of right-invariant differential operators homogeneous of degree $k$.

   Let $n\in \N$.
    We denote by $C^\infty_{c,n}(\mathfrak{g})$ the subspace of $u\in C^\infty_c(\mathfrak{g})$ such that $\int Pu=0$ for any polynomial $P\in C^\infty(\mathfrak{g})$ which is homogeneous of degree $-m-Q$ with $m<n$.
    By convention $C^\infty_{c,0}(\mathfrak{g}):=C^\infty_c(\mathfrak{g})$.

   
\begin{prop}\label{prop:sum_D_inv} 
    For any $n\in \N$, $u\in C^\infty_{c,n}(\mathfrak{g})$ if and only if $u=\sum_i \delta_{D_i}\ast v_i$ where 
  $v_i\in C^\infty_c(\mathfrak{g})$ and    $D_i\in \DO^{n_i}(\mathfrak{g})$ is homogeneous of degree $n_i\geq n$.
\end{prop}
\begin{proof}
    The implication $\impliedby$ is trivial. Let us prove the other implication.
 The following lemma is classical.
  \begin{lem}\label{lem:aux_4}
      If $u\in C^\infty_c(\R^n)$ with $\int_{\R^n} u=0$, 
      then there exists $u_1,\cdots,u_n\in C^\infty_c(\R^n)$ such that $u=\sum_{i=1}^n\frac{\partial u_i}{\partial x_i}$.
  \end{lem}
  \begin{lem}\label{lem:aux_5}
      For each $g\in \mathfrak{g}$, let $Y_g$ the constant vector field on $\mathfrak{g}$ which is equal to $g$, and 
      $X_g$ the right-invariant vector field which is equal to $g$ at $0$. If $g\in \mathfrak{g}^i$, then $X_g-Y_g$ is a sum of elements of the form $PY_h$, where $h\in \mathfrak{g}^j$ with $j> i$,
      $P$ a polynomial on $\mathfrak{g}$ which only depends on $\oplus_{k=1}^{j-1}\mathfrak{g}^k$. In particular, $Y_h(P)=0$.
  \end{lem}
  \begin{proof}
      This follows immediately from the BCH formula, because $X_g(h)=\frac{d}{dt}\Bigr|_{t=0} (tg)\cdot h=\sum_{k=0}^{+\infty}c_k\ad_{h}^k(g)$ for some coefficients $c_k\in \Q$.
  \end{proof}
  Let $u\in C^\infty_{c,1}(\mathfrak{g})$. 
  \Cref{lem:aux_4} implies that $u$ is a sum of elements of the form $Y_g(v)$ with $v\in C^\infty_c(\mathfrak{g})$ and $g\in \mathfrak{g}^i$ for some $i$.
  By \Cref{lem:aux_5}, we can replace $Y_g(v)$ by $X_g(v)$ and the sum of elements of the form $PY_h(v)=Y_h(Pv)$ with $h\in \mathfrak{g}^j$ for $j>i$.
  By iterating this procedure, the result follows.
  For general $n\in \N$, we proceed by induction.
  Let $n\geq 2$ and $u\in  C^\infty_c(\mathfrak{g})$ such that $\int Pu=0$ for all 
  $P\in C^\infty(\mathfrak{g})$ polynomial which satisfies $P\circ \alpha_\lambda=\lambda^m P$ for some $m<n$. 
  By applying the induction hypothesis, we can write $u$ as sum of elements of the form $D(v)$ with $D\in \DO^{m}(\mathfrak{g})$ with $m\geq n-1$.
  We fix $w\in C^\infty_c(\mathfrak{g})$ such that $\int w=1$. We write each term $D(v)$ as $D(v-\langle 1,v\rangle w)+\langle 1,v\rangle D(w)$.
  By the case $n=1$ applied to $v-\langle 1,v\rangle w$, we can write $u$ as sum of an element of the form $D(v)$ with $D$ homogeneous of degree $\geq n$ and elements of the form $D(w)$ with $D$ homogeneous of degree $\leq n-1$. We now fix a basis of $\oplus_{m\leq n-1}\DO^{m}(\mathfrak{g})$. 
    Since this space is dual to the space polynomials which satisfy $P\circ \alpha_\lambda=\lambda^m P$ for some $m<n$, we deduce that for $\langle P,u\rangle=0 $ to vanish for all such $P$, all the terms
  of the form $D(w)$ have to vanish.
\end{proof}
The space $C^\infty_{c,n}(\mathfrak{g})$ is clearly closed under taking adjoint, because if $P$ is homogeneous of degree $-m-Q$, then $P^*$ is also homogeneous of degree $-m-Q$. 
So, in \Cref{prop:sum_D_inv}, one can replace $\delta_D\ast u$ by $u\ast \delta_D$.
From this, one deduces the following:
          \begin{equation}\label{eqn:adjoint_Cnc}\begin{aligned}
           C^\infty_{c,n}(\mathfrak{g})^*=C^\infty_{c,n}(\mathfrak{g}),\quad             C^\infty_{c,n}(\mathfrak{g})\ast C^\infty_{c,m}(\mathfrak{g})\subseteq C^\infty_{c,n+m}(\mathfrak{g}),\quad \forall n,m\in \Z_+.
         \end{aligned}\end{equation}

\paragraph{Unitary representations and distributions:} 
Let $\pi$ be a unitary representation of $G$, $C^\infty(\pi)\subseteq L^2(\pi)$ the subspace of smooth vectors, i.e., the subset of $\xi\in L^2(\pi)$ such that the map $g\in G\mapsto \pi(g)\xi \in L^2(\pi)$ is Fréchet smooth. 
Let $D$ be a right-invariant differential operator on $\mathfrak{g}$, $\xi \in C^\infty(\pi)$. We define $\pi(D)\xi \in L^2(\pi)$ to be $D(g\mapsto\pi(g)\xi)$ evaluated at the $g=0$. 
If $\xi$ is smooth, then $\pi(D)\xi$ is smooth.
So, $\pi(D)\in \mathrm{End}(C^\infty(\pi))$.
 One can directly check that 
    $\pi(D_1D_2)=\pi(D_1)\pi(D_2)$
 for all $D_1,D_2$ right-invariant differential operators.
We equip $C^\infty(\pi)$ with the topology generated by all the seminorms 
$       \norm{\pi(D)\xi}_{L^2(\pi)}$.
This topology makes $C^\infty(\pi)$ a Fréchet space. 
We denote by $ C^{-\infty}(\pi)$ the space of topological anti-linear dual of $C^\infty(\pi)$ equipped with the weak topology.
\begin{ex}\label{ex:left_regular}
    Let $\Theta$ be the left-regular representation of $G$ on $L^2(G)$.
   By local Sobolev embedding theorems, $C^\infty(\Theta)$ is the space of $f\in C^\infty(G)$ such that for any $D\in \DO(\mathfrak{g})$, $D(f)\in L^2(G)$.
   This example will be of particular importance in the end of the article.
\end{ex}

\paragraph{Extension to distributions:}
Let $\xi\in C^\infty(\pi)$ and $u\in C^{-\infty}_c(\mathfrak{g})$.
If $\eta\in L^2(\pi)$, then $g\in G\mapsto \langle \eta,\pi(g)\xi\rangle\in \C$ is a smooth function.
We define $\pi(u)\xi\in L^2(\pi)$ by applying the Riesz representation theorem to the functional
    \begin{equation*}\begin{aligned}
        \langle \eta,\pi(u)\xi\rangle:=\langle g\mapsto \langle \pi(g)\xi,\eta\rangle ,u\rangle,\quad \forall \eta\in L^2(\pi)
    \end{aligned}\end{equation*}
The fact that the functional $\eta\mapsto \langle g\mapsto \langle \pi(g)\xi,\eta\rangle ,u\rangle$ is bounded follows immediately from continuity of $u$ as an anti-linear functional $C^\infty_c(\mathfrak{g})\to \C$. One can check that
    \begin{equation*}\begin{aligned}
        \pi(\delta_D)=\pi(D),\quad \pi(u)C^\infty(\pi)\subseteq C^\infty(\pi),\quad        \pi(u_1\ast u_2)=\pi(u_1)\pi(u_2),\quad \langle \pi(u)\xi,\eta\rangle=\langle \xi,\pi(u^*)\eta\rangle,
    \end{aligned}\end{equation*}
    where $\xi,\eta\in C^\infty(\pi)$ and $D$ is a right-invariant differential operator.
The map $C^\infty(\pi)\xrightarrow{\pi(u)} C^\infty(\pi)$ is continuous. By duality, it extends to a continuous map $C^{-\infty}(\pi)\xrightarrow{\pi(u)} C^{-\infty}(\pi)$.
If $u\in C^\infty_c(\mathfrak{g})$, then $\pi(u)C^{-\infty}(\pi)\subseteq C^\infty(\pi)$. Furthermore,
    \begin{equation}\label{eqn:L2pi}\begin{aligned}
         \pi(u)\xi=\int_\mathfrak{g}u(g)\pi(g)\xi\odif{g},\quad \forall \xi\in L^2(\pi), u\in C^{\infty}_c(\mathfrak{g})
    \end{aligned}\end{equation}
 which implies that 
 \begin{equation}\label{eqn:identities_pi_of_composition_norm_of_pi_adjoint}
  \norm{\pi(u)}_{\cL(L^2(\pi))}\leq \norm{u}_{L^1}.
\end{equation}

  \paragraph{$C^*$-calculus:} Let $k\in \C$. The space $\Rd{k}$ denotes the space of distributions $u\in  C^{-\infty}_c(\mathfrak{g})$ which satisfy \begin{equation}\label{eqn:homog_cond_ek}
 \alpha_\lambda(u)-\lambda^ku\in C^\infty_c(\mathfrak{g}),\quad \forall \lambda\in \Rpt.
\end{equation}
For example $\delta\in\Rd{0}$, and if $D\in \DO^k(\mathfrak{g})$, then $\delta_D\in \Rd{k}$.
One has
    \begin{equation*}\begin{aligned}
      \Rd{k}\ast \Rd{l}\subseteq  \Rd{k+l},\quad \Rd{k}^*=\Rd{\overline{k}},\quad \forall k,l\in \C.        
    \end{aligned}\end{equation*}
If $u\in \Rd{k}$, then $\mathrm{singsupp}(u)\subseteq \{0\}$. To see this, notice that by \eqref{eqn:homog_cond_ek}, the singular support of $u$ is a compact subset of $\mathfrak{g}$ which is invariant under $\Rpt$ action on $\mathfrak{g}$. So, $\singsupp(u)\subseteq \{0\}$.
\begin{prop}\label{prop:intDS}
    Let $k\in\C$, $n\in \Z_+$ such that $\Re(k)<n$.
    If $f\in C^\infty_{c,n}(\mathfrak{g})$, then 
        \begin{equation}\label{eqn:sqdklfmjqlskdfjqlsdjmfsd}\begin{aligned}
          u=  \int_0^1 \alpha_{t}(f)\frac{\odif{t}}{t^{1+k}}\in \cA_k        
        \end{aligned}\end{equation}
    Conversely, if $u\in \cA_k$, there exists $f\in C^\infty_{c,n}(\mathfrak{g})$ and $g\in C^\infty_c(\mathfrak{g})$ such that 
            \begin{equation}\label{eqn:sum_Ds}\begin{aligned}
                u=g+        \int_0^1 \alpha_{t}(f)\frac{\odif{t}}{t^{1+k}}.
            \end{aligned}\end{equation}
\end{prop}
The condition $\Re(k)<n$ basically ensures that the integral defines a distribution.
\begin{proof}
    We can suppose that $f=\delta_D\ast g$ with $D\in \DO^m(\mathfrak{g})$ and $g\in C^\infty_c(\mathfrak{g})$ with $m>\Re(k)$.
    We have $\int_0^1 \alpha_{t}(f)\frac{\odif{t}}{t^{1+k}}=\delta_D\ast \int_0^1 \alpha_{t}(g)\frac{\odif{t}}{t^{1+k-m}}$.
    The second integral converges in $L^1$ by \eqref{eqn:norm_dilations}. So, $u$ is a well-defined compactly supported distribution which satisfies \eqref{eqn:homog_cond_ek}.
    We now prove the converse. 
    \begin{lem}\label{lem:smoothness}
        For any $k\in \C$, $u\in \cA_k$, the map \begin{equation}\label{eqn:qiosdjfm}
 \Psi:\Rpt\to C^\infty_c(\mathfrak{g}),\quad \Psi(\lambda)=\alpha_\lambda(u)-\lambda^ku.
\end{equation}
is smooth.
    \end{lem}
    \begin{proof}

    We follow \cite[Proof of Lemma 21]{ErikBobCalculus}. 
  Let $m\in \N$ be fixed,
    $B_n:=\{\lambda\in \Rpt:\norm{\Psi(\lambda)}_{C^m}\leq  n\}$.
  Clearly $\Rpt=\bigcup_{n\in \N}B_n$. Since $\Psi$ is continuous as a map $\Rpt\to  C^{-\infty}(\mathfrak{g})$, it follows that $B_n$ is closed for any $n\in \N$. 
  By the Baire category theorem, there exists $n$ such that $B_n$ has non empty interior. 
  By the cocycle relation $\Psi(\lambda\lambda')=\alpha_{\lambda}(\Psi(\lambda'))+\lambda^{\prime k}\Psi(\lambda)$, we deduce that the map 
          $\lambda\in \Rpt\mapsto \norm{\Psi(\lambda)}_{C^m}\in \R$
      is locally bounded. Continuity of $\Psi:\Rpt\to C^\infty_c(\mathfrak{g})$  now follows from the Arzelà-Ascoli theorem

  We now show that $\Psi:\Rpt\to C^\infty_c(\mathfrak{g})$ is smooth. Let $F\in C^\infty_c(\Rpt)$ such that 
  $\int_{0}^{+\infty} F(t)t^k\frac{dt}{t}=1$. 
    One can check that
       \begin{equation*}\begin{aligned}
            \Psi(\lambda)=\int_{0}^{+\infty} F(\lambda^{-1}t)\Psi(t)-F(t)\alpha_\lambda(\Psi(t)) \frac{dt}{t}.        
       \end{aligned}\end{equation*}
    Since $\Psi:\Rpt\to C^\infty_c(\mathfrak{g})$ is continuous, it follows that  $\Psi:\Rpt\to C^\infty_c(\mathfrak{g})$ is smooth. 
    \end{proof}
    Let $u\in \cA_k$.
    By replacing $u$ with $u-g$ for some $g\in C^\infty_c(\mathfrak{g})$ such that $\langle P,g\rangle=\langle P,u\rangle$ for all polynomials $P$ homogeneous of degree $-m-Q$ with $m\leq n$, we can suppose that $\langle P,u\rangle=0$.
    By \Cref{lem:smoothness}, 
        \begin{equation*}\begin{aligned}
            f:=\odv{\alpha_t(u)-t^ku}{t}_{t=1}\in C^\infty_{c,n}(\mathfrak{g}).
        \end{aligned}\end{equation*}
     The chain rule applied to the identity $\alpha_{ts}(u)=\alpha_{t}(\alpha_s(u))$ implies that 
       \begin{equation*}\begin{aligned}
                    \odv{t^{-k}\alpha_t(u)}{t}=\frac{\alpha_t(f)}{t^{k+1}}
       \end{aligned}\end{equation*}
    The integral  $\int_0^1\frac{\alpha_t(f)}{t^{k+1}}$ defines an element of $\cA_k$ because $f\in C^\infty_{c,n}(\mathfrak{g})$.
    So, the limit 
        \begin{equation*}\begin{aligned}
            \lim_{t\to 0^+}t^{-k}\alpha_t(u)=u-\int_0^1\frac{\alpha_t(f)}{t^{k+1}}
        \end{aligned}\end{equation*}
    exists (as a distribution), and it belongs to $\cA_k$.
    By \eqref{eqn:norm_dilations}, the limit is supported in $\{0\}$.
    If $k\notin \Z_+$ this implies that the limit vanishes. If $k\in \Z_+$, then the limit is of the form $\delta_D$ for some $D\in \DO^k(\mathfrak{g})$.
    It also vanishes in this case by our assumption on the vanishing of $\langle P,u\rangle$.
\end{proof}



\begin{theorem}\label{thm:u_bounded} If $k\in \C$ with $\Re(k)\leq 0$ and $u\in \Rd{k}$ and $\pi$ a unitary representation, then $\pi(u)L^2(\pi)\subseteq L^2(\pi)$ and $L^2(\pi)\xrightarrow{\pi(u)}L^2(\pi)$ is continuous.

\end{theorem}
To my knowledge, this theorem is due to Knapp and Stein \cite{KnappStein1}, see also \cite{GoodmanBoundedL2}. Our proof is basically the same as that of Knapp and Stein. We only simplified some technical details.
\begin{proof}
We start with the case $\Re(k)<0$. Let $v=u-2^{k}\alpha_{2^{-1}}(u)\in C^\infty_c(\mathfrak{g})$. 
We have $$u=2^{k(n+1)}\alpha_{2^{-k(n+1)}}(u)+\sum_{j=0}^{n}2^{kj}\alpha_{2^{-j}}(v).$$ 
By \eqref{eqn:limit_delta_0},
\begin{equation*}\begin{aligned}
        \pi(u)\xi=\sum_{j=0}^{\infty}2^{kj}\pi(\alpha_{2^{-j}}(v))\xi,\quad \forall \xi\in C^{-\infty}(\pi). 
    \end{aligned}\end{equation*}
If $\xi\in L^2(\pi)$, then the sum converges in $L^2(\pi)$ because by \eqref{eqn:norm_dilations} and \eqref{eqn:identities_pi_of_composition_norm_of_pi_adjoint}, we have $$\sum_{j=0}^{\infty}|2^{kj}|\norm{\pi(\alpha_{2^{-j}}(v))\xi}_{L^2(\pi)}\leq \sum_{j=0}^\infty 2^{\Re(k)j}\norm{\alpha_{2^{-j}}(v)}_{L^1}\norm{\xi}_{L^2(\pi)}= \sum_{j=0}^\infty 2^{\Re(k)j}\norm{v}_{L^1}\norm{\xi}_{L^2(\pi)}.$$

We now treat the case $\Re(k)=0$. By replacing $u$ with $u^*\ast u$, we can suppose that $k=0$. The proof in this case uses the following lemma. 
\begin{lem}[{Cotlar-Stein Lemma \cite[Lemma 18.6.5]{HormanderIII}}]Let $(A_n)_{n\in \N}$ be a sequence of bounded operators acting on a Hilbert space $H$, such that $$A=\sup_{j}\sum_{k=1}^{+\infty}\norm{A_j^*A_k}^{1/2}_{\cL(H)}<+\infty,\quad B=\sup_{j}\sum_{k=1}^{+\infty}\norm{A_jA_k^*}^{1/2}_{\cL(H)}<+\infty.$$ The sum $\sum_{k=1}^{+\infty}A_k$ converges in the strong topology to a bounded operator whose norm is $\leq \sqrt{AB}$.
 \end{lem}
\begin{lem}\label{lem:qjsdofjo} 
  There exists a continuous semi-norm $q$ on $C^\infty_c(\mathfrak{g})$ such that 
      \begin{equation*}\begin{aligned}
          \norm{\sum_{j\in \Z} \beta_j\pi(\alpha_{2^j}(v))}_{\cL(L^2(\pi))}\leq q(v)
      \end{aligned}\end{equation*}
  for all $v\in C^\infty_{c,1}(\mathfrak{g})$ and $(\beta_j)_{j\in \Z}\in \C$ such that $|\beta_j|\leq 1$ for all $j\in \Z$.
\end{lem}
\begin{proof}
We apply the Cotlar-Stein lemma. By replacing $\pi$ with $\bigoplus_{j\in \Z}\pi\circ \alpha_{2^j}$, we can suppose that 
$\norm{\pi(\alpha_{2^j}(\phi))}_{\cL(L^2(\pi))}=\norm{\pi(\phi)}_{\cL(L^2(\pi))}$ forall $\phi\in C^\infty_c(\mathfrak{g})$.
Hence, 
    \begin{equation*}\begin{aligned}
        \norm{\pi\Big(\overline{\beta_j} \beta_l\alpha_{2^j}(v)^*\ast\alpha_{2^l}(v)\Big)}_{\cL(L^2(\pi))}\leq \norm{\pi\Big(v^*\ast\alpha_{2^{l-j}}(v)\Big)}_{\cL(L^2(\pi))}.
    \end{aligned}\end{equation*}
 So, it suffices to show that there exists a continuous semi-norm $q$ on $C^\infty_c(\mathfrak{g})$ such that
$$\sum_{j\in \Z}\norm{\pi\Big(v^*\ast \alpha_{2^j}(v)\Big)}^\frac{1}{2}_{\cL(L^2(\pi))}\leq q(v),\quad \sum_{j\in \Z}\norm{\pi\Big(v\ast \alpha_{2^j}(v^*)\Big)}^\frac{1}{2}_{\cL(L^2(\pi))}\leq q(v)$$
By symmetry, we only deal with the first.
We divide the sum $\sum_{j\in \Z}=\sum_{j\geq 0}+\sum_{j< 0}$.
Since $v\in C^\infty_{c,1}(\mathfrak{g})$, we can write $v$ as a finite sum of elements of the form $\delta_D\ast w$ where $w\in C^\infty_c(\mathfrak{g})$ and $D\in \DO^n(\mathfrak{g})$ with $n\geq 1$. 
Following the proof of \Cref{prop:sum_D_inv}, after choosing a linear basis $\oplus_{j=1}^{N}\DO^j(\mathfrak{g})$, one can choose the $w$ to depend continuously on $v$.
So, after replacing $v$ with $\delta_D\ast w$,  we have \begin{align*}
\sum_{j<0}\norm{\pi\Big(v^*\ast \alpha_{2^j}(\delta_D\ast w)\Big)}_{\cL(L^2(\pi))}^\frac{1}{2}&=\sum_{j< 0}2^{\frac{jn}{2}}\norm{\pi\Big(v^*\ast \delta_D\ast  \alpha_{j}(w)\Big)}_{\cL(L^2(\pi))}^\frac{1}{2}\\&\leq 
\sum_{j< 0}2^{\frac{jn}{2}}\norm{v^*\ast \delta_D}_{L^1}^\frac{1}{2}\norm{w}_{L^1}^\frac{1}{2}.
\end{align*}
 and
 \begin{align*}
\sum_{j\geq 0}\norm{\pi\Big((\delta_D \ast w)^*\ast\alpha_{2^j}(v)\Big)}_{\cL(L^2(\pi))}^\frac{1}{2}&=\sum_{j\geq 0}\norm{\pi\Big( w^*\ast \delta_{D^*}\ast\alpha_{2^j}(v)\Big)}_{\cL(L^2(\pi))}^\frac{1}{2}\\
&=\sum_{j\geq 0}2^{\frac{-jn}{2}}\norm{\pi\Big( w^*\ast \alpha_{2^j}(\delta_{D^*}\ast v)\Big)}^\frac{1}{2}_{\cL(L^2(\pi))}\\
&\leq \sum_{j\geq 0}2^{\frac{-jn}{2}}\norm{w^*}_{L^1}^\frac{1}{2}\norm{\delta_{D^*}\ast v}_{L^1}^\frac{1}{2}.\qedhere
\end{align*}
\end{proof}
Let $u\in \Rd{0}$, $v=u-\alpha_{2^{-1}}(u)\in C^\infty_{c}(\mathfrak{g})$. 
By Proposition \ref{prop:sum_D_inv}, $v\in C^\infty_{c,1}(\mathfrak{g})$. 
By \eqref{eqn:limit_delta_0}, we deduce that
    \begin{equation}\label{eqn:aux_11}\begin{aligned}
        \pi(u)-\langle 1,u\rangle\pi(\delta)=\sum_{j\leq 0}\pi(\alpha_{2^j}(v)).
    \end{aligned}\end{equation}
Since $\pi(\delta)=\Id$, the result follows.
\end{proof}
It follows from \Cref{thm:u_bounded} that we can define 
    \begin{equation*}\begin{aligned}
        \norm{u}_{\overline{\cA}_0}:=\sup_{\pi\in \hat{G}}\norm{\pi(u)}_{\cL(L^2(\pi))}=\Xi(u)_{L^2(G)},\quad u\in \Rd{0}.  
        ,      
    \end{aligned}\end{equation*}
where $\Xi$ is the left regular representation of $\mathfrak{g}$ seen as a Lie group, and $\hat{G}$ is the space of irreducible unitary representations of $G$. The second equality follows from the fact that the Lie group $G$ is amenable.
We have thus defined a $C^*$-calculus.
There are two types of representations of the calculus $(\cA_k)_{k\in \C}$ that are of interest to us.
\paragraph{Representations of first type:}
If $\pi$ is a unitary representation of $G$, then we define $\pi_k:=\pi$ for all $k\in \C$.
The operators $\pi(u)\in \cL(C^\infty(\pi))$ are well-defined for any $u\in \cA_k$ because $u$ is compactly supported.
The topology on $C^\infty(\pi)$ defined using the semi-norms $\norm{\pi(u)\xi}_{L^2(\pi)}$ for $k\in \C$ and $u\in \cA_k$ coincides with the usual topology on $C^\infty(\pi)$ because for any $D\in \DO^n(\mathfrak{g})$, $\delta_D\in \cA_n$.
A consequence of the above is that we have defined Sobolev spaces $H^k(\pi)$ for $k\in \R$, see the paragraph following \Cref{dfn:unitary_representation_abstract}.
\paragraph{Representations of second type:}
Let $\pi$ be a \textit{non-trivial irreducible} unitary representation of $G$. We define $\pi_k$ by the formula
 by the formula 
    \begin{equation}\label{eqn:pi_k_Helffer}\begin{aligned}
        \pi(\sigma^k(u))\xi:=\lim_{n\to +\infty}2^{-nk}\pi(\alpha_{2^n}(u))\xi,\quad \forall u\in \cA_k,\xi\in C^{\infty}(\pi)
    \end{aligned}\end{equation}
We will use the notation $ \pi(\sigma^k(u))$ instead of $\pi_k$ to avoid confusion between representations of first and second type.
\begin{theorem}\label{thm:pi_k_Helffer}
     Limit \eqref{eqn:pi_k_Helffer} exists in $C^{\infty}(\pi)$, and the induced map $C^\infty(\pi)\xrightarrow{\pi(\sigma^k(u))} C^\infty(\pi)$ is continuous.
    The family $(\pi(\sigma^k(\cdot)))_{k\in \C}$ defines a representation   of $(\cA)_{k\in \C}$. In fact, the following hold: 
    \begin{enumerate}
         \item  If $u\in\cA_k,w\in\cA_l,\xi,\eta \in C^{\infty}(\pi)$, then   \begin{equation*}\begin{aligned}
            \pi_0(\delta)=\mathrm{Id}_{C^\infty(\pi)},\quad    \pi(\sigma^k(u))\pi(\sigma^l(w))=\pi_{k+l}(uw),\quad \langle \xi,\pi(\sigma^k(u))\eta\rangle=\langle \pi(\sigma^{\bar{k}}(u^*))\xi,\eta\rangle.        
            \end{aligned}\end{equation*}
        \item If $\Re(k)\leq 0$, then $\pi(\sigma^k(u))$ extends to a bounded operator on $L^2(\pi)$. Furthermore, for all $u\in \cA_0$, $\norm{\pi(\sigma^0(u))}_{\cL(L^2(\pi))}\leq \norm{u}_{\overline{\cA_0}}$.
        \item For any $k\in \C$ and $u\in \cA_k$, $\pi(u)-\pi(\sigma^k(u))$ extends to a continuous map $C^{-\infty}(\pi)\to C^\infty(\pi)$. 
       In particular, the topology on $C^\infty(\pi)$ defined using the semi-norms $\norm{\pi(\sigma^k(u))\xi}_{L^2(\pi)}$ for $k\in \C$ and $u\in \cA_k$ coincides with the usual topology on $C^\infty(\pi)$, and the Sobolev spaces defined using the representation $(\pi(\sigma^k(\cdot)))_{k\in \C}$ coincide with the Sobolev spaces defined using the representation $\pi$ seen as a representation of first type.
     \end{enumerate}  
\end{theorem}
\begin{proof}
    By the identity $\pi(\sigma^{k+l}(\delta_D\ast u))=\pi(\delta_D)\ast \pi(\sigma^k(u))$ where $D\in \DO^l(\mathfrak{g})$, it follows that to show that Limit \eqref{eqn:pi_k_Helffer} exists in $C^\infty(\pi)$, it suffices to show that it exists in $L^2(\pi)$ and that 
    the induced $C^\infty(\pi)\xrightarrow{\pi(\sigma^k(u))} L^2(\pi)$ is continuous.
    To this end, let $v=2^{-k}\alpha_{2}(u)-u\in C^\infty_c(\mathfrak{g})$. 
    So,
        \begin{equation}\label{eqn:qsjkodfjmqsdf}\begin{aligned}
                 2^{-nk}\pi(\alpha_{2^n}(u))=\pi(u)+\sum_{j=0}^{n-1}2^{-jk}\pi( \alpha_{2^j}(v))   
        \end{aligned}\end{equation}
 The result now follows from the following variant of the Riemann--Lebesgue lemma:
    \begin{lem}\label{lem:Riemannian}
        If $\pi$ is a non-trivial irreducible unitary representation, and $u\in C^\infty_c(\mathfrak{g})$, then for any $l\in \N$, $\lim_{n\to +\infty}2^{-nl}\norm{\pi(\alpha_{2^n}(u))}_{\cL(L^2(\pi))}=0$.
    \end{lem}
    \begin{proof}
        Since the trivial representation isn't weakly contained in $\pi$, there exists $v\in C^\infty_c(\mathfrak{g})$ such that $\norm{\pi(v)}<|\langle 1,v\rangle|$.
   By rescaling, we can suppose that $\langle 1,v\rangle=1 $.
   So, $\norm{\pi(v^*\ast v)}<1$ and the operator $\mathrm{Id}_{L^2(\pi)}-\pi(v):L^2(\pi)\to L^2(\pi)$ is invertible. 
    By \eqref{eqn:norm_dilations} and \eqref{eqn:identities_pi_of_composition_norm_of_pi_adjoint}, 
   \begin{equation*}\begin{aligned}
           \norm{ \pi(\alpha_{2^n}(u))(\mathrm{Id}-\pi(v))}_{\cL(L^2(\pi))}= \norm{ \pi(\alpha_{2^n}(u)-\alpha_{2^n}(u)\ast v)}_{\cL(L^2(\pi))}&\leq \norm{\alpha_{2^n}(u)-\alpha_{2^n}(u)\ast v}_{L^1}\\
          &=\norm{u-u\ast \alpha_{2^{-n}}(v)}_{L^1}
        \end{aligned}\end{equation*}
       We have
        \begin{equation}\label{eqn:sdkqjfljkqsodfjmoqsdjfmoqksd}\begin{aligned}
              u(g)-u\ast \alpha_{2^{-n}}(v)(g)
              &=\int \big(u(g)-u(g\alpha_{2^{-n}}(h)^{-1})\big)v(h)\odif{h}
            \end{aligned}\end{equation}
    The mean value theorem together with compactness of the support of $u,v$ implies that 
                        \begin{equation*}\begin{aligned}
                            \limsup_{n\to +\infty}2^n\norm{u-u\ast \alpha_{2^{-n}}(v)}_{L^1}<+\infty.
                        \end{aligned}\end{equation*}
        Since $\mathrm{Id}-\pi(v)$ is invertible, we deduce that        
         \begin{equation}\label{eqn:aux_10}\begin{aligned}
          \limsup_{n\to +\infty}2^n\norm{\pi(\alpha_{2^n}(u))}_{\cL(L^2(\pi))}<+\infty
      \end{aligned}\end{equation}
    Now, there are two methods to proceed. Either one does the same calculation but replacing $\mathrm{Id}-\pi(v)$ by $(\mathrm{Id}-\pi(v))^l$ and using higher order mean value theorem.
    Or one could use the Dixmier-Mallivain theorem which we recall below
    \begin{theorem}[{\cite[3.1 Théorème]{DixmierMalliavin}, see also \cite{DixmierMalliavinLieGroupoids}}]\label{thm:Dixmier-Malliavin}
      Let $G$ be a Lie group equipped with a Haar measure. Every $u\in C^\infty_c(G)$
 can be expressed as $v_1\ast w_1 +\cdots+ u_m\ast v_m$ for some $m\in \N$,  $v_i, w_i \in  C^\infty_c(G)$.
    \end{theorem}
    By an iterated application of the Dixmier-Mallivain theorem and \eqref{eqn:aux_10}, we deduce that for any $l\in \N$,
        \begin{equation*}\begin{aligned}
          \limsup_{n\to +\infty}2^{nl}\norm{\pi(\alpha_{2^n}(u))}_{\cL(L^2(\pi))}<+\infty
        \end{aligned}\end{equation*}
        This finishes the proof of \Cref{lem:Riemannian}.
    \end{proof}
    Notice that $\pi(\sigma^k(u))-\pi(u)=\sum_{j=0}^{+\infty}2^{-jk}\pi( \alpha_{2^j}(v)) $ which by \Cref{lem:Riemannian} maps $C^{-\infty}(\pi)$ continuously to $C^\infty(\pi)$.
    In particular $\sum_{j=0}^{+\infty}2^{-jk}\pi( \alpha_{2^j}(v))$ is a bounded operator $L^2(\pi)\to L^2(\pi)$.
    If furthermore $\Re(k)\leq 0$, then $\pi(\sigma^k(u))$ extends to a bounded operator on $L^2(\pi)$ by \Cref{thm:u_bounded}.
    
    If $k=0$, then 
        \begin{equation*}\begin{aligned}
         \norm{\pi(\alpha_{2^n}(u))}_{\cL(L^2(\pi))}\leq \norm{\alpha_{2^n}(u)}_{\overline{\cA_0}}=\norm{u}_{\overline{\cA_0}},\quad \forall n\in \N.
        \end{aligned}\end{equation*}
    So, $\norm{\pi(\sigma^0(u))}_{\cL(L^2(\pi))}\leq \norm{u}_{\overline{\cA_0}}$.

    We now show that $\pi(\sigma^k(u))\pi(\sigma^l(w))=\pi_{k+l}(uw)$. If $\xi,\eta\in C^\infty(\pi)$, then we have
        \begin{equation*}\begin{aligned}
            \langle \pi(\sigma^k(u))\pi(\sigma^l(w))\xi,\eta\rangle=\langle \pi(\sigma^l(w))\xi,\pi(\sigma^{\bar{k}}(u^*))\eta\rangle&=\lim_{n\to +\infty}\langle 2^{-nl}\pi(\alpha_{2^n}(w))\xi,2^{-n\overline{k}}\pi(\alpha_{2^n}(u^*))\eta\rangle   \\&
            =\lim_{n\to +\infty}\langle 2^{-n(l+k)}\pi(\alpha_{2^n}(uw))\xi,\eta\rangle\\&=\langle \pi_{k+l}(uw)\xi,\eta\rangle.   
        \end{aligned}\end{equation*}
        The rest of \Cref{thm:pi_k_Helffer} is straightforward to prove.
    \end{proof}
    \begin{rem}\label{rem:extension_at_0}
        If $u\in \cA_0$, then for \textit{any} unitary representation $\pi$, and for any $\xi\in L^2(\pi)$, Limit \eqref{eqn:pi_k_Helffer} exists in $L^2(\pi)$ and defines a bounded operator on $L^2(\pi)$.
       This follows immediately from \MyCref{thm:u_bounded} and \eqref{eqn:qsjkodfjmqsdf}.
    \end{rem}

 \begin{rem}\label{rem:vanishing_smooth}
       For any $k\in \C$, if $u\in C^\infty_c(\mathfrak{g})$, then $\pi(\sigma^k(u))=0$ by \Cref{lem:Riemannian}.
   \end{rem}
Let \begin{itemize}
    \item   $\Sigma_1(\cA)$ be the set\footnote{One should take equivalence classes of unitary representations to avoid set-theoretic issues. We will ignore this.} of all representations of $(\cA)_{k\in \C}$ of first type which come from irreducible unitary representations of $G$
    \item   $\Sigma_2(\cA)$ be the set of all representations of $(\cA)_{k\in \C}$ of second type which come non-trivial irreducible unitary representation of $G$.
\end{itemize}
We take $\Sigma(\cA)=\Sigma_1(\cA)\sqcup \Sigma_2(\cA)$.

        \begin{theorem}\label{thm:completion}
           The $C^*$-calculus $(\cA)_{k\in \C}$ with $\Sigma(\cA)$ satisfy the hypotheses of \Cref{thm:simple_arg_sobolev}.           
        \end{theorem}
        \begin{proof}
            Let $\Theta$ be the left regular representation of $G$, see \Cref{ex:left_regular}.
           %
            Let $B$ be the Hausdorff completion of $\cA_0$ by 
                \begin{equation*}\begin{aligned}
                    \norm{u}_B:=\norm{\Theta(\sigma^0(u))},
                \end{aligned}\end{equation*}
                where we used \MyCref{rem:extension_at_0}.
            Notice that $B$ is a Hausdorff completion, because if $u\in C^\infty_c(\mathfrak{g})$, then $\norm{u}_B=0$ by \Cref{rem:vanishing_smooth}.
           We have an obvious short sequence 
               \begin{equation}\label{eqn:sqkdmfjsqdjkfqmjsdf}\begin{aligned}
                    0\to C^*G \to \overline{\cA_0}\to B\to 0       
               \end{aligned}\end{equation}
            It is clearly exact at $C^*G$ and $B$. It is also exact in the middle.
            To see this, let $u\in \cA_0$.
            We will show that $\norm{u}_{\overline{\cA_0}/C^*G}\leq 2\norm{u}_B$ which implies that \eqref{eqn:sqkdmfjsqdjkfqmjsdf} is exact.
            Let $f\in C^\infty_c(\mathfrak{g})$ such that $\int f=1$.
            So, 
                \begin{equation*}\begin{aligned}
                 \Theta(u-u\ast \alpha_{2^{-n}}(f))=\Theta(\sigma^0(u))(\mathrm{Id}-\Theta(\alpha_{2^{-n}}(f)))- (\Theta(\sigma^0(u))-\Theta(u))(\mathrm{Id}-\Theta(\alpha_{2^{-n}}(f)))
                \end{aligned}\end{equation*}
            The first term is bounded in norm by $2\norm{u}_B$. So, we only need to show that the second converges in norm to $0$. In fact, by \eqref{eqn:qsjkodfjmqsdf}, if $v=\alpha_2(u)-u\in C^\infty_c(\mathfrak{g})$, then
                \begin{equation*}\begin{aligned}
                            \norm{(\Theta(\sigma^0(u))-\Theta(u))(\mathrm{Id}-\Theta(\alpha_{2^{-n}}(f)))}_{\cL(L^2(G))}\!&=\norm{\Theta\left(\sum_{m=0}^{+\infty}\alpha_{2^m}(v)-\alpha_{2^m}(v)\ast \alpha_{2^{-n}}(f)\right)}_{\cL(L^2(G))}\\
                                                                                                                                                    &\leq \sum_{m=0}^{+\infty}\norm{\alpha_{2^m}(v)-\alpha_{2^m}(v)\ast \alpha_{2^{-n}}(f)}_{L^1}\\
                                                                                                                                                    &= \sum_{m=n}^{+\infty}\norm{v-v\ast \alpha_{2^{-m}}(f)}_{L^1}
                            \end{aligned}\end{equation*}
            By \eqref{eqn:sdkqjfljkqsodfjmoqsdjfmoqksd} and mean value theorem, there exists $C>0$ such that $\norm{\alpha_{2^m}(v)-\alpha_{2^m}(v)\ast f}_{L^1}\leq C2^{-m}$ for all $m\geq 0$.

            To finish the proof of the first two hypotheses of     \Cref{thm:simple_arg_sobolev},
            we need to show that $B$ is of Type I and that its irreducible representations are all equivalent to $\pi(\sigma^0(\cdot))$ for some $\pi$ non-trivial unitary representation of $G$.
            This result  is due to Fermanian-Kammerer and Fischer \cite{FischerDefect}. Our argument is different from theirs. 
            Ours is basically the same as that of Ewert \cite{ewert2021pseudodifferential}.

            For the convenience of the reader we will sketch Ewert's argument.
            Let $C^*_{1}G$ be the completion of $C^\infty_{c,1}(\mathfrak{g})$ in $C^*G$.
            We will show that $B$ is Morita equivalent to $C^*_1 G\rtimes \Rpt$, where $\Rpt$ acts on $C^*_1 G$ by dilations.
            In fact, $B$ is a typical example of Rieffel proper action, see \cite{RieffelProperActions,BussEchterhoffRieffel}.
            Once one has the Morita equivalence, the result follows from Kirillov's orbit method applied to the group $G\rtimes \Rpt$ which is an exponential solvable group, see \cite{LudwigBook}.

            We will use the theory of Hilbert $C^*$-modules and multipliers, see \cite{LanceBook}.
            On the space $C^\infty_{c,1}(\mathfrak{g})$, by \Cref{prop:intDS}, we define the $B$-valued inner product 
                \begin{equation*}\begin{aligned}
                    \langle f,g\rangle=\int_0^1 \alpha_t(f^*\ast g)\frac{\odif{t}}{t}\in B,\quad f,g\in C^\infty_{c,1}(\mathfrak{g}).        
                \end{aligned}\end{equation*}
            This is an inner product. The only non-trivial part is to show that if $u\in \cA_0$, then $\langle f,g\ast u\rangle=\langle f,g\rangle\ast u$.
            In fact, 
                \begin{equation*}\begin{aligned}
                    \langle f,g\ast u\rangle-\langle f,g\rangle\ast u\in C^\infty_c(\mathfrak{g}).        
                \end{aligned}\end{equation*}
            This follows by writing $u$ as in \eqref{eqn:sum_Ds}. Details are left to the reader.
            Let $E$ be the completion of $C^\infty_{c,1}(\mathfrak{g})$ under the norm induced by this inner product, i.e., $\norm{f}_E=\norm{\langle f,f\rangle}_B^{1/2}$.
            The space $E$ is a right-$C^*$-Hilbert $B$-module.
            The Hilbert $E$ module is full, i.e., $\langle C^\infty_{c,1}(\mathfrak{g}),C^\infty_{c,1}(\mathfrak{g})\rangle$ is dense in $B$.
            To see this, let $u\in \cA_0$. Take $n>N$ where $N$ is the depth of the Lie group $G$.
            By \Cref{prop:intDS},  there exists $f\in C^\infty_{c,n}(\mathfrak{g})$ such that $u-\int_0^1\alpha_t(f)\frac{\odif{t}}{t}\in C^\infty_c(\mathfrak{g})$.
            The result follows from the following lemma:
            \begin{lem}
                Any $f\in  C^\infty_{c,n}(\mathfrak{g})$ can be written as sum of elements of the form $f_1\ast f_2$ with $f_1,f_2\in C^\infty_{c,1}(\mathfrak{g})$.
                So, $u$ modulo $C^\infty_c(\mathfrak{g})$ is equal to a sum of elements of the form $\langle f_1^*,f_2\rangle$.
            \end{lem}
            \begin{proof}
                 By \Cref{prop:sum_D_inv}, we can suppose that $f=\delta_{D}\ast f'$ where $D$ is homoegneous of degree $\geq n$.
            Since $n>N$, $D$ isn't a vector field, so we can suppose that $D=D_1D_2$ where $D_1$ and $D_2$ homogeneous of degree $\geq 1$.
            So, $\delta_{D_2}\ast f'\in C^\infty_{c,1}(\mathfrak{g})$. Since $C^\infty_{c,1}(\mathfrak{g})$ is closed under taking adjoints, we can write $\delta_{D_2}\ast f'$ as a sum of elements of the form $h\ast \delta_{D''}$.
            For some $D''$ homogeneous of degree $\geq 1$.
            To summarize, $f$ can be written as a sum of elements of the form $\delta_{D_1}\ast f'\ast \delta_{D_2}$ with $D_1$ and $D_2$ homogeneous of degree $\geq 1$.
            The result follows from \Cref{thm:Dixmier-Malliavin} applied to $f'$.
            \end{proof}
            Since the module $E$ is full, it gives a Morita equivalence between compact operators (in the sense of Kasparov) $\cK(E)$ and $B$.
            By the definition of compact operators on $C^*$-modules, $\cK(E)=C^*_1 G\rtimes \Rpt$.

        We now construct the elements $v_k\in \cA_k$ as in \Cref{thm:simple_arg_sobolev}, see \Cref{rem:K_smaller}.
        \begin{lem}\label{lem:inj}
            There exists $f\in  C^\infty_c(\mathfrak{g})$ and $f'\in C^\infty_{c,1}(\mathfrak{g})$ such that $f,f'$ are positive as elements of the $C^*$-algebra $C^*G$, and
            \begin{enumerate}
                \item\label{lem:inj:1}     If $\pi$ is a unitary representation of $G$, then $\pi(f):L^2(\pi)\to L^2(\pi)$ is injective.
                \item  If $\pi$ is a non-trivial irreducible unitary representation of $G$, then $\pi(f'):L^2(\pi)\to L^2(\pi)$ is injective.
            \end{enumerate}
        \end{lem}
        \begin{proof}
            Let $h\in C^\infty_c(\mathfrak{g})$ be any smooth function such that $\int h=1$.
           We define
            \begin{equation*}\begin{aligned}
                    f=\sum_{n=1}^{+\infty} c_n \alpha_{2^{-n}}(h^*\ast h)   
            \end{aligned}\end{equation*}
        where $c_n>0$ is any sequence which converges to $0$ fast enough that $f\in C^\infty(\mathfrak{g})$.
        The function $f$ is compactly supported by \eqref{eqn:norm_dilations}.
        
        Let $\pi$ be any unitary representation of $G$, $\xi\in L^2(\pi)$.
        If $\pi(f)\xi=0$, then $\pi(\alpha_{2^{-n}}(h))\xi=0$ for all $n\in \N$. Since $\pi(\alpha_{2^{-n}}(h))\xi\to (\int h)\xi=\xi$ in $L^2(\pi)$ for any $\xi\in L^2(\pi)$, see \eqref{eqn:L2pi}, it follows that $\xi=0$.
        
         Let $X_1,\cdots,X_k$ be right-invariant vector fields that linearly span $\mathfrak{g}$ at the origin.
        We take 
            \begin{equation*}\begin{aligned}
                f'=\sum_{i=1}^n(\delta_{X_i}\ast f)^*\ast  (\delta_{X_i}\ast f).        
            \end{aligned}\end{equation*}
        Let $\pi$ be a non-trivial irreducible unitary representation, $\xi \in L^2(\pi)$ such that $\pi(f')\xi=0$.
        So, $\pi(\delta_{X_i}\ast f)\xi=0$ for all $i$.
        The vector $\eta=\pi(f)\xi$ is a smooth vector, and $\pi(\delta_{X_i})\eta=0$.
        Hence, $\eta$ is an invariant vector. Since $\pi$ is a non-trivial irreducible representation, $\eta=0$.
        So, $\xi=0$.
\end{proof}
        For any $k\in \C$ with $\Re(k)\in [-\frac{1}{2},\frac{1}{2}]$, let
            \begin{equation}\label{eqn:v_k}\begin{aligned}
                    v_k:=f+\int_0^1 \alpha_t(f') \frac{\odif{t}}{t^{1+k}}.  
            \end{aligned}\end{equation}
        By \Cref{prop:intDS}, $v_k\in \cA_k$.
        If $\pi\in \Sigma_1(\cA)$, then injectivity of $\pi(v_0)$ on $L^2(\pi)$ comes from injectivity of $\pi(f)$.
        If $\pi\in \Sigma_2(\cA)$, then injectivity of 
            \begin{equation*}\begin{aligned}
                \pi(\sigma^0(v_0))=\int_0^{+\infty} \alpha_t(f') \frac{\odif{t}}{t}
            \end{aligned}\end{equation*}
         on $L^2(\pi)$ comes from injectivity of $\pi(f')$.
        So, the operator $v_0$ is invertible in $\overline{\cA_0}$ by \Cref{prop:simple_arg}.
        Continuity of $k \mapsto v_{-k} a v_{k}\in \overline{\cA_0}$ for any $a\in \cA_0$ follows immediately from \Cref{lem:qjsdofjo}.
    \end{proof}

\paragraph{Theorem of Helffer, Nourrigat and Glowski:}
By the results, we proved in this section, we can apply \Cref{thm:simple_arg_sobolev} to the calculus $(\cA_k)_{k\in \C}$.
We now show that the following theorem due Glowski \cite{Glowacki2} which is based on a theorem of Helffer and Nourrigat \cite{HelfferRockland}.
\begin{theorem}\label{thm:Helffer_Nourrigat}
Let $k\in \C$ and $u\in \cA_k$. The following are equivalent
\begin{enumerate}
\item For every $\pi\in \Sigma_2(\cA)$, the maps $C^\infty(\pi)\xrightarrow{\pi(\sigma^k(u))} C^\infty(\pi)$  and $C^\infty(\pi)\xrightarrow{\pi(\sigma^{\bar{k}}(u^*))} C^\infty(\pi)$ are injective.
\item For every $\pi\in \Sigma_2(\cA)$, the maps 
    \begin{equation}\label{eqn:skqdjfkjqsdjfmlqjsdjfmlqksdf}\begin{aligned}
        &\pi(\sigma^k(u)):C^{-\infty}(\pi)\to C^{-\infty}(\pi)\\
&\pi(\sigma^k(u)):H^l(\pi)\to H^{l-\Re(k)}(\pi)&&\forall l\in \R^\nu\\
&\pi(\sigma^k(u)):C^\infty(\pi)\to C^\infty(\pi)
    \end{aligned}\end{equation}
    are topological isomorphisms.
    \item There exists $v\in \Rd{-k}(\mathfrak{g})$ such that $v\ast u-\delta$ and $u\ast v- \delta$ belong to $C^\infty_c(\mathfrak{g})$
\end{enumerate}
\end{theorem}
\begin{proof}
    The implication $3\implies 2$ follows from \Cref{rem:vanishing_smooth}.
    The implication $2\implies 1$ is trivial. 
    We now prove $1\implies 3$.
    By symmetry, it suffices to show that $u$ has a left parametrix in $\cA_{-k}$.
    By replacing $u$ by $v_{k_1}\cdots v_{k_n}\ast u^*\ast u$ where $v_{k_i}\in \cA_{k_i}$ are as in \eqref{eqn:v_k} and $k_1,\cdots,k_n\in [-\frac{1}{2},\frac{1}{2}]$ such that $k_1+\cdots+k_n=-2\Re(k)$, we can suppose that $k=0$.
    Let $f\in C^\infty_c(\mathfrak{g})$ as in \Crefitem{lem:inj}{1}.
    By replacing $u$ by $u^*u+f$, we can suppose that $u\in \cA_0$ satisfies $u=u^*$, and for any irreducible unitary representation, $\pi(u)$ is injective on $C^\infty(\pi)$, and if $\pi$ is non-trivial, then $\pi(\sigma^0(u))$ is also injective on $C^\infty(\pi)$.
    In particular, \Cref{thm:simple_arg_sobolev} applies to $u$.

    Let $\Theta$ be the left-regular representation, see \Cref{ex:left_regular}.    
    By \Cref{thm:simple_arg_sobolev}, $\Theta(u)$ is a topological isomorphism of $C^\infty(\Theta)$, and of $L^2(G)$.
    Let $T:L^2(G)\to L^2(G)$ be its inverse. From the inclusion $C^\infty_c(\mathfrak{g})\subseteq C^\infty(\Theta)\subseteq C^\infty(\mathfrak{g})$, we deduce that $T(C^\infty_c(\mathfrak{g}))\subseteq C^\infty(\mathfrak{g})$.
    If $R_g:L^2(G)\to L^2(G)$ denotes the right translation by $g\in \mathfrak{g}$, then $\Theta(u)$ commutes with $R_g$.
    So, $T$ commutes with $R_g$.
    It follows that there exists a unique distribution $v\in C^{-\infty}(\mathfrak{g})$ such that $T(f)=v\ast f$ for all $f\in C^\infty_c(\mathfrak{g})$.
    Hence, $v\ast u=\delta=u\ast v$.

    We now show that the singular support of $v$ is $\{0\}$.
    To this end, for any $f\in C^\infty_c(\mathfrak{g})$, let $M_f:L^2(G)\to L^2(G)$ be the multiplication operator.
    It maps $C^\infty(\Theta)$ to $C^\infty(\Theta)$. So, by duality, it is well-defined $C^{-\infty}(\Theta)\to C^{-\infty}(\Theta)$.
    \begin{lem}\label{lem:commutator}
       For any $k\in \C$, $u\in \cA_k$, $v_1,\cdots,v_n\in C^\infty_c(\mathfrak{g})$,   the iterated commutator 
           \begin{equation*}\begin{aligned}
                [\cdots[\Theta(u),M_{v_1}],M_{v_2}],\cdots, ],M_{v_n}] \in \cL(C^{-\infty}(\Theta)).      
           \end{aligned}\end{equation*}
       maps the Sobolev space $H^{l}(\Theta)$ to $H^{l-k+n}(\Theta)$
    \end{lem}
    \begin{proof}
        In this proof, we let the dilation $\alpha_t$ act on functions on $\mathfrak{g}\times \mathfrak{g}$ by the formula $\alpha_t(f)(x,y)=f(\alpha_{t^{-1}}(x),\alpha_{t^{-1}}(y))t^{-Q}$, see \eqref{eqn:dilation_Q}.

        We need to show that if $u'\in \cA_{n-k}$, then $\Theta(u')[\cdots[\Theta(u),M_{v_1}],M_{v_2}],\cdots, ],M_{v_n}]$ is a bounded operator on $L^2(G)$.
        By \Cref{prop:intDS} and \Cref{prop:sum_D_inv}, we can suppose that $u=\int_0^1\alpha_t(D(f))\frac{\odif{t}}{t^{1+k}}$ for some $f\in C^\infty_{c}(\mathfrak{g})$, $D$ is left-invariant homogeneous of degree $>\Re(k)$.
        The iterated commutator $[\cdots[\Theta(u),M_{v_1}],M_{v_2}],\cdots, ],M_{v_n}]$ has a Schwartz kernel $K(x,y)=\int_0^1\alpha_t(h_t)\frac{\odif{t}}{t^{1+k}}$ where
            \begin{equation}\label{eqn:qsdiofjisdqojfioqsdfqsd}\begin{aligned}
                 h_t(x,y)=D(f)(xy^{-1})\Big(v_1(\alpha_t(x))-v_1(\alpha_t(y))\Big)\cdots \Big(v_n(\alpha_t(x))-v_n(\alpha_t(y))\Big)
            \end{aligned}\end{equation}
        Similarly, by \Cref{prop:intDS} and \Cref{prop:sum_D_inv}, we can suppose that $u'=\int_{0}^1 \alpha_t(D'(f'))\frac{\odif{t}}{t^{1+n-k}}$, where $f'\in C^\infty_c(\mathfrak{g})$ and $D'$ is right-invariant homogeneous of degree $>-\Re(k)+n$.
        So, the operator $\Theta(u')[\cdots[\Theta(u),M_{v_1}],M_{v_2}],\cdots, ],M_{v_n}]$ has a kernel which is equal to by a simple change of variables 
            \begin{equation*}\begin{aligned}
					\int_0^1\alpha_t\left(\int_0^1 \int \alpha_s(D'(f'))(xy^{-1})\frac{h_t(y,z)}{t^n} \odif{y}\frac{\odif{s}}{s^{1+n-k}}+\int_0^1\int D'(f')(xy^{-1})\alpha_s(\frac{h_t}{t^n})(y,z)\odif{y}\frac{\odif{s}}{s^{1+k}}\right)\frac{\odif{t}}{t}
            \end{aligned}\end{equation*}      
        We observe that the function 
            \begin{equation*}\begin{aligned}
                    K'_t(x,y)=        \int_0^1 \int \alpha_s(D'(f'))(xy^{-1})\frac{h_t(y,z)}{t^n} \odif{y}\frac{\odif{s}}{s^{1+n-k}}+\int_0^1\int D'(f')(xy^{-1})\alpha_s(\frac{h_t}{t^n})(y,z)\frac{\odif{s}}{s^{1+k}}\odif{y}
            \end{aligned}\end{equation*}
        is a smooth function in $(x,y,t)\in \mathfrak{g}\times \mathfrak{g}\times \R_+$.
        Here, smoothness in $x,y$ comes from the degree of homogenity of $D$ and $D'$, while smoothness in $t$ is obvious from \eqref{eqn:qsdiofjisdqojfioqsdfqsd}.
        Furthermore $K''_t$ has the property that it is equal to $D'_x K''_t(x,y)$ and $D''_yK'''_t(x,y)$ where $K''$ and $K'''$ are smooth functions on $\mathfrak{g}\times \mathfrak{g}\times \R_+$, and $D_x'$ means that $D'$ acts on the $x$ variable, while 
        and $D''_y$ means that $D^t$ acts on the $y$ variable, where $D^t(l):=D(l(y^{-1}))$. 
        This is enough so that the Cotlar Stein method (the integral version) applies exactly like in \Cref{lem:qjsdofjo} and that the operator whose kernel is $\int_0^1\alpha_t(K'_t)\frac{\odif{t}}{t}$ is bounded.
    \end{proof}
    Let us show that the singular support of $v$ is $\{0\}$. We need to show that $u$ is hypoelliptic, i.e., 
        \begin{equation*}\begin{aligned}
            \singsupp(u\ast f)=\singsupp(f),\quad \forall f\in C^\infty_c(\mathfrak{g}).       
        \end{aligned}\end{equation*}
    Notice that $C^\infty_c(\mathfrak{g})\subseteq C^{-\infty}(\Theta)$.
    We will show that $u$ is maximally hypoelliptic, i.e., for any open set $U\subseteq \mathfrak{g}$
        \begin{equation}\label{eqn:qksjfdjsqdjfmsqjdflmqsdf}\begin{aligned}
            u\ast f\in H^{s}(\Theta) \text{ locally on }U \implies f\in H^{s}(\Theta) \text{ locally on }U ,\quad \forall    s\in \R,f\in H^s(\Theta).          
        \end{aligned}\end{equation}
    We recall that a distribution $w$ belongs to $H^s(\Theta)$ by locally on $U$ means that for any cutoff function $\chi\in C^\infty_c(U)$, $\chi w\in H^s(\Theta)$.
    Let us prove \eqref{eqn:qksjfdjsqdjfmsqjdflmqsdf}.
    Notice that by \Cref{thm:simple_arg_sobolev}, \eqref{eqn:qksjfdjsqdjfmsqjdflmqsdf} is satisfied with $U=\mathfrak{g}$, i.e., we have global maximal hypoellipticity.
    Now, let $f\in C^{-\infty}(\Theta)$ such that $u\ast f\in H^{s}(\Theta) \text{ locally on }U$, let $K\subseteq U$ be a compact subset.
    We choose a sequence of cutoff functions $\chi_n\in C^\infty_c(U)$ such that $\chi_n\chi_{n+1}=\chi_{n+1}$, and $\chi_n=1$ on $K$ for all $n$.
    We will show that for some $n$, $\chi_nf\in H^s(\Theta)$.
    To this end, let $t\in \R$ such that $f\in H^t(\Theta)$.
    In the following equations, we will write $[u,\chi_n]$ instead of $[\Theta(u),M_{\chi_n}]$ for simplicity.
    So, 
        \begin{equation*}\begin{aligned}
            u\ast (\chi_1f)=[u,\chi_1](f)+\chi_1 (u\ast f)        
        \end{aligned}\end{equation*}
    By \Cref{lem:commutator}, the right hand side belongs to $H^{\min(t+1,s)}$.
    Hence, by global maximal hypoellipticity, $\chi_1f\in H^{\min(t+1,s)}$.
    Now, 
        \begin{equation*}\begin{aligned}
            u\ast (\chi_2f)&=u\ast (\chi_2\chi_1f)=[u,\chi_2](\chi_1f)+\chi_2 (u\ast \chi_1 f)  =     [u,\chi_2](\chi_1f)+\chi_2 ([u,\chi_1](f)+\chi_1 (u\ast f)     )\\
&=[u,\chi_2](\chi_1f)+[\chi_2 ,[u,\chi_1]](f)+[u,\chi_1](\chi_2f)+\chi_2 (u\ast f)     
        \end{aligned}\end{equation*}
    The first term is in $H^{\min(t+1,s)+1}$, the second in $H^{t+2}$, the third in $H^{\min(t+1,s)+1}$ because $\chi_2f=\chi_2 \chi_1f\in H^{\min(t+1,s)}$, the fourth in $H^s$.
    So, $\chi_2f\in H^{\min(t+2,s)}$.
    Iterating this procedure leads to $\chi_nf\in H^{\min(t+n,s)}$. Maximal hypoellipticity follows and $\singsupp(v)=\{0\}$.

   Let $w=\chi v$ where $\chi\in C^\infty_c(\mathfrak{g})$ is equal to $1$ in a neighbourhood of $0$.
   So, $w\ast u-\delta\in C^\infty_c(\mathfrak{g})$.
   By taking the adjoint and using the fact that $u^*=u$, it follows that $u\ast w^{*}-\delta\in C^\infty_c(\mathfrak{g})$.
   Putting the two equations together, we deduce that $w-w^{*}\in C^\infty_c(\mathfrak{g})$. So, $w$ is a two-sided parametrix of $u$.
   Finally, one has 
       \begin{equation*}\begin{aligned}
            w-\alpha_\lambda(w)=w\ast (\alpha_\lambda(u)-u)\ast \alpha_\lambda(w) \mod C^\infty_c(\mathfrak{g}).        
       \end{aligned}\end{equation*}
   Hence, $w-\alpha_\lambda(w)\in C^\infty_c(\mathfrak{g})$. So, $w\in \cA_0$.
    \end{proof}
    \begin{rem}
        Let us explain a little, what is going on in the proof of \MyCref{thm:Helffer_Nourrigat}. Our result, \MyCref{thm:simple_arg_sobolev}[3] gives global maximal hypoellipticity. To construct the parameterix, one needs local maximal hypoellipticity. 
        To get this, one needs the calculus to contain cutoff functions and that the commutator with cutoff functions decreases the order by $1$. Once one has this, the argument (which is quite standard) following \MyCref{lem:commutator} gives local maximal hypoellipticity from global maximal hypoellipticity.
        
        What we could have done is from the very beginning define a calculus which contains both $M_f$ for $f\in C^\infty_c(\mathfrak{g})$ as operators of order $0$ and $\cA_k$ as operators of order $k$, and then prove that in this calculus taking commutator with $M_f$ reduces the order by $1$.
        This calculus was defined in \cite{ChrGelGloPol}. We didn't take this approach as we thought it complicates the rest of this section. Instead, we took the approach of proving (in a very adhoc manner) the required commutator condition (\MyCref{lem:commutator}).

        In our previous work \cite{MohsenMaxHypo}, we didn't have this difficulty because our calculus already contained multiplication by smooth functions, and we proved that taking the commmutator with multiplication operators reduces the order by $1$.
    \end{rem}


\begin{refcontext}[sorting=nyt]
\printbibliography
\end{refcontext}
\end{document}